\documentclass[11pt]{article}

\usepackage{amssymb}
\usepackage{graphics}

\newcommand{\qed}{$\;\;\;\Box$}
\newenvironment{proof}{\par\smallbreak{\sl Proof.~}}
{\unskip\nobreak\hfill \qed \par\medbreak}

\newcounter{claim}
\renewcommand{\theclaim}{\arabic{claim}}
\newenvironment{claim}{\refstepcounter{claim}%
\par\medskip\par\noindent{\bf Claim~\theclaim.}\rm}%
{\par\medskip\par}

\newcounter{cclaim}
\renewcommand{\thecclaim}{\arabic{cclaim}}
{\par\medskip\par}

\newenvironment{subproof}{\par\noindent{\it Proof of Claim.}}%
{$\Box$\par\smallbreak}
\newcommand{\hide}[1]{}

\newcommand{\DD}{{\cal D}}

\newcommand{\N}{{\mathbb N}}

\newcommand{\R}{{\mathbb R}}
\newcommand{\C}{{\mathbb C}}
\newcommand{\Z}{{\mathbb Z}}

\newcommand{\uu}{\mbox{\boldmath $u$\unboldmath}}
\newcommand{\vv}{\mbox{\boldmath $v$\unboldmath}}
\newcommand{\ww}{\mbox{\boldmath $w$\unboldmath}}
\newcommand{\A}{{\cal A}}
\newcommand{\B}{{\cal B}}

\newcommand{\PP}{{\cal P}}

\newcommand{\CC}{{\cal C}}
\newcommand{\J}{{\cal J}}

\newcommand{\F}{{\cal F}}

\newcommand{\W}{{\cal W}}

\newcommand{\RR}{{\cal R}}

\newcommand{\LL}{{\cal L}}
\newcommand{\K}{{\cal K}}
\newcommand{\KK}{{\cal K}}
\newcommand{\EE}{{\cal E}}
\newcommand{\UU}{{\cal U}}
\newcommand{\beq}{\begin{equation}}
\newcommand{\ee}{\end{equation}}

\renewcommand{\d}{\partial}

\newtheorem{thm}{Theorem}
\newtheorem{lemma}[thm]{Lemma}

\newtheorem{defn}[thm]{Definition}

\newtheorem{rem}[thm]{Remark}

\newcommand{\de}{\delta}
\newcommand{\eps}{\varepsilon}
\newcommand{\vphi}{\varphi}
\newcommand{\la}{\lambda}
\newcommand{\vp}{\varphi}
\newcommand{\om}{\omega}

\newcommand{\reff}[1]{(\ref{#1})}

\renewcommand{\Im}{\mathop{\mathrm{Im}}\nolimits}
\renewcommand{\Re}{\mathop{\mathrm{Re}}\nolimits}

\title{
Hopf Bifurcation for General 1D  Semilinear \\  Wave Equations
with Delay} 

\newcounter{thesame}
\author{
	Irina Kmit
	\thanks{Institute of Mathematics, Humboldt University of Berlin. On leave from the
		Institute for Applied Problems of Mechanics and Mathematics,
		Ukrainian National Academy of Sciences. 
		Tel.: 
	+49(030)2093-45380.
		{\small   E-mail:
			{\tt kmit@mathematik.hu-berlin.de}
	}}
	\ \ \ Lutz Recke \thanks{Institute of Mathematics, Humboldt University of Berlin.
		{\small   E-mail:
			{\tt recke@mathematik.hu-berlin.de}}
}}

\date{}

\begin{document}

\maketitle

\begin{abstract}
\noindent 
We consider boundary value problems for 1D autonomous damped and delayed
semilinear wave equations of the type
$$
\d^2_tu(t,x)- a(x,\la)^2\d_x^2u(t,x)= b(x,\la,u(t,x),u(t-\tau,x),\d_tu(t,x),\d_xu(t,x)), \; x \in (0,1)
$$
with smooth coefficient functions $a$ 
and $b$
such that $a(x,\la)>0$ and
$b(x,\la,0,0,0,0) = 0$ for all $x$ and $\la$.
We state conditions ensuring Hopf bifurcation, i.e.,
 existence, local uniqueness (up to time shifts), regularity (with respect to $t$ and $x$)
and smooth dependence (on $\tau$ and $\la$) of small non-stationary time-periodic
solutions, which bifurcate from the stationary solution $u=0$, and we
derive a formula which determines the bifurcation direction with respect to the bifurcation parameter~$\tau$.

To this end, we transform the wave equation into a system of partial  integral equations
by means of integration along characteristics, and then we apply
a Lyapunov-Schmidt procedure and a generalized implicit function
theorem
 to this system.
The main technical difficulties, which have to be managed, are typical for hyperbolic PDEs
(with or without delay):
small divisors and the ``loss of derivatives'' property.

We do not use any properties of the corresponding initial-boundary value problem.
In particular, our results are true also for negative delays $\tau$.

\end{abstract}



\section{Introduction}\label{sec:intr}
\renewcommand{\theequation}{{\thesection}.\arabic{equation}}
\setcounter{equation}{0}

\subsection{The problem}\label{sec:problem}
This paper concerns  1D autonomous damped and delayed semilinear wave equation of the general type 
\beq\label{eq:1.1}
\d_t^2u(t,x)- a(x,\la)^2\d^2_xu(t,x)=b(x,\la,u(t,x),u(t-\tau,x),\d_tu(t,x),\d_xu(t,x)),\;x\in(0,1),
\ee
with 
one Dirichlet and one Neumann boundary condition
\beq\label{eq:1.2}
u(0,t) = \d_xu(t,1)=0.
\ee
It is supposed that $b(x,\la,0,0,0,0)=0$ for all $x$ and $\la$, i.e., that $u=0$ is a stationary solution to
\reff{eq:1.1}--\reff{eq:1.2}  for all $\tau$ and $\la$.

The goal is to describe Hopf bifurcation, i.e., existence and local uniqueness (up to time shifts)
of families (parametrized by $\tau$ and $\la$) of non-stationary time-periodic
solutions to \reff{eq:1.1}--\reff{eq:1.2}, which bifurcate from the stationary solution $u=0$.

Our main result, stated in Theorem \ref{thm:hopf} below, is quite similar to Hopf bifurcation theorems 
for delayed ODEs (see, e.g.,  \cite{Erneux}, \cite[Chapter 5.5]{Guo/Wu}, \cite[Chapter 11]{Hale/Lunel}, \cite{Moiola,Sieber}) 
and for delayed parabolic PDEs (see, e.g.,  \cite{Chen,Du,Faria,LiMa}, \cite[Chapter 6]{Wu}).
However, the analysis 
of Hopf bifurcation for hyperbolic PDEs is faced with considerable complications if compared to ODEs
or parabolic PDEs  (with or without delay).
In the present paper we provide an approach for
overcoming the following technical difficulties, which appear in dissipative hyperbolic PDEs
and  do not appear in  ODEs or parabolic PDEs:

First, the question, whether a nondegenerate time-periodic solution to a dissipative nonlinear
wave equation is locally unique (up to time shifts in the autonomous case) and whether it depends smoothly on the system
parameters, is much more delicate 
than for ODEs or parabolic PDEs 
(cf., e.g., \cite{Raugel,Raugel1}). One reason for that
is the so-called loss of derivatives for  hyperbolic PDEs.
To overcome this difficulty, we use a generalized implicit function theorem
\cite[Theorem 2.2]{KR4}, which is applicable to
abstract equations with a  loss of derivatives
property.
Remark that for smoothness of the data-to-solution map of hyperbolic PDEs it is necessary, in general,
that the equation depends smoothly not only on the data and on the unknown function $u$, but also on the space variable $x$
(and the time variable $t$ in the non-autonomous case).
This is completely different to what is known for parabolic PDEs (cf. \cite{GR}).

Second, analysis of time-periodic solutions to hyperbolic PDEs usually encounters a complication
known as the problem of small divisors \cite{Arn,iooss_s,vejvoda}. 
Since Hopf bifurcations can be expected only in the so-called non-resonant case,
where small divisors do not appear, we have to impose a condition (assumption \reff{Fred} below)
preventing  small divisors from coming up.
That condition has no counterparts 
in the case of ODEs or parabolic PDEs.

And third, linear autonomous hyperbolic PDEs with one space dimension differ essentially from those with more than one space dimension:
They satisfy the spectral mapping property (see \cite{Lopes} in $L^p$-spaces and, more important
for applications to nonlinear problems, \cite{Lichtner1} in $C$-spaces)
and they generate Riesz bases (see, e.g., \cite{Guo,Joly}), what is not the case,
in general, if the space dimension is larger than one (see the celebrated counter-example of M. Renardy in \cite{Renardy}).
Therefore the question of  Fredholmness of the corresponding differential operators in appropriate spaces
of time-periodic functions is highly difficult.

The main consequence (from the point of view of mathematical techniques) of the fact, that the space
dimension of  \reff{eq:1.1}, \reff{eq:1.2} is one, consists in the following: We can use integration along characteristics
in order to replace \reff{eq:1.1}, \reff{eq:1.2} by an nonlinear partial integral equation (see \cite{Appell1} for the notion
``partial integral equation'').
After that, we can apply
known Fredholmness properties to the linearized partial integral equation (\cite{KRsecond}, \cite[Corollary 4.11]{KR4})
and, hence, we can apply the Lyapunov-Schmidt reduction method to the nonlinear  partial integral equation.

\subsection{Main results}\label{sec:results}
Our goal is to investigate time-periodic solutions to \reff{eq:1.1}--\reff{eq:1.2}. In order to work in spaces
of functions with  fixed time period $2\pi$, we put the frequency parameter $\om$ explicitely
into the equation by  scaling the time variable $t$ and by introducing
a new unknown function $u$ 
as follows:
$$
u_{\rm new}(t,x):=u_{\rm old}\left(\frac{t}{\om},x\right).
$$
The problem  \reff{eq:1.1}--\reff{eq:1.2} for the new unknown function $u$ and the unknown frequency $\om$ reads
\beq
\label{problem}
\left.
\begin{array}{l}
\om^2\d_t^2u(t,x)- a(x,\la)^2\d_x^2u(t,x)=b(x,\la,u(t,x),u(t-\om\tau,x),\om\d_tu(t,x),\d_xu(t,x)),\\ 
u(t,0) = \d_xu(t,1)=0,\\
u(t+2\pi,x)=u(t,x).
\end{array}
\right\}
\ee
Throughout this paper we suppose (and we do not mention it further) that
$$
\begin{array}{l}
a:[0,1]\times \R \to \R \mbox{ and } b:[0,1]\times \R^5 \to \R \mbox{ are $C^\infty$-smooth,}\\
a(x,\la)>0  \mbox{ and } b(x,\la,0,0,0,0)=0 \mbox{ for all } x \in [0,1] \mbox{ and } \la \in \R.
\end{array}
$$

Assumptions $\bf(A1)$--$\bf(A3)$ below 
are standard for Hopf bifurcation. To formulate them,
we consider the following eigenvalue problem for the linearization of \reff{problem} in $u=0$, $\om=1$ and $\la=0$:
\beq\label{evp}
\left.
\begin{array}{l}
  \left(\mu^2-b^0_5(x)\mu-b^0_4(x)e^{-\mu\tau}-b^0_3(x)\right)u(x)=a_0(x)^2u''(x)+b^0_6(x)u'(x),\\
u(0) = u'(1)=0.
\end{array}
\right\}
\ee
Here $\mu \in \C$ and $u:[0,1] \to \C$ are eigenvalue and eigenfunction, respectively. The coefficients $a_0$ and $b^0_j$ in
\reff{evp} are defined by
\beq
\label{bdef}
a_0(x):=a(x,0),\;  b^0_j(x):=\d_jb(x,0,0,0,0,0) \mbox{ for } j=3,4,5,6,
\ee
where $\d_jb$ is the partial derivative of the function $b$ with respect to its $j$th variable.\\

Our first assumption states that for certain delay $\tau=\tau_0$ there exists a pair of pure imaginary geometrically simple
eigenvalues to \reff{evp} (without loss of generality we may assume that the pair is $\mu=\pm i$):

{\bf(A1)} There exists $\tau_0\in\R$ such that for $\mu=i$ and $\tau=\tau_0$ there exists exactly one (up
to linear dependence) solution  $u\ne 0$ to  (\ref{evp}).\\

The second assumption is the so-called nonresonance condition:

{\bf(A2)} If   $u\not=0$ is a solution to (\ref{evp})  with   $\mu=ik, k \in \Z$  and  $\tau=\tau_0$,  
then  $k=\pm 1$.\\

The third assumption is the so-called transversality 
condition with respect to change of parameter~$\tau$. It states that for all $\tau \approx \tau_0$
there exists exactly one eigenvalue $\mu=\hat{\mu}(\tau)\approx i$ to \reff{evp}
and that this eigenvalue crosses the imaginary axis transversally if $\tau$ crosses  $\tau_0$.
In order to formulate this more explicitly, we consider the adjoint  problem
to \reff{evp} with $\mu=i$ and  $\tau=\tau_0$:
\beq\label{ad}
\left.
\begin{array}{l}
  \left(-1+ib^0_5(x)-b^0_4(x)e^{i\tau_0}-b^0_3(x)\right)u(x)
  =(a_0(x)^2u(x))''-(b^0_6(x)u(x))',\\
u(0) = a_0(1)^2u'(1)+(2a_0(1)a_0'(1)-b_6^0(1))u(1)=0.
\end{array}
\right\}
\ee
Because of assumption {\bf(A1)} there exists exactly one (up
to linear dependence) solution  $u\ne 0$ to (\ref{ad}).
The transversality condition is the following:

{\bf(A3)} For any solution $u=u_0\ne 0$ to  \reff{evp}  with $\tau=\tau_0$ and  $\mu=i$ and
for any solution $u=u_*\ne 0$ to  \reff{ad} it holds
$$
\sigma:=\int_0^1\left(2i-b^0_5+\tau_0e^{-i\tau_0}b^0_4\right)u_0\overline{u_*}dx \ne 0,\quad
\rho:=\mbox{Im}\left(\frac{e^{-i\tau_0}}{\sigma}\displaystyle\int_0^1b^0_4u_0\overline{u_*}dx\right)\ne 0.
$$
Remark that $\mbox{Re}\,\hat{\mu}'(\tau_0)=\rho$,
and this real number does not depend on the choice of the eigenfunctions $u_0$ and $u_*$.
The complex number $\sigma$ depends on the choice of the eigenfunctions $u_0$ and $u_*$, but the fact,
if condition $\sigma\not= 0$ is satisfied or not, does not depend on this choice.

\begin{defn}\label{def:Cspaces}
(i) We denote by $C_{2\pi}(\R\times [0,1])$ the space of all continuous functions $u:[0,1] \times \R \to \R$ such that
$u(t+2\pi,x)=u(t,x)$ for all $t \in \R$ and $x \in [0,1]$, with the norm 
$$
\|u\|_\infty:=\max\{|u(t,x)|: t \in \R, \; x \in [0,1]\}.
$$

(ii) For $k \in \N$  we denote by $C^k_{2\pi}(\R\times [0,1])$ the space of all $C^k$-smooth 
$u\in C_{2\pi}(\R\times [0,1])$, with the norm 
$
\max\{\|\d^i_t\d^j_xu\|_\infty: 0 \le i+j\le k\}.
$
\end{defn}

Now we are prepared to formulate our Hopf bifurcation theorem.
\begin{thm}
\label{thm:hopf}
Suppose that conditions $\bf(A1)$--$\bf(A3)$ are fulfilled as well as
\beq
\label{Fred}
\int_0^1\frac{b^0_5(x)}{a_0(x)}dx \ne 0.
\ee
Let  $u=u_0 \ne 0$ be a  solution to  \reff{evp}  with $\tau=\tau_0$ and  $\mu=i$,
and let  $u=u_* \ne 0$ be a  solution to  \reff{ad}.
Then there exist $\eps_0>0$ and a $C^\infty$-map
$$
(\hat u,\hat\om,\hat\tau) : [0,\eps_0]\times [-\eps_0,\eps_0]\to C^2_{2\pi}(\R \times [0,1])\times \R^2
$$
such that the following is true:

(i) {\rm Existence:} For all $(\eps,\la)\in (0,\eps_0]\times [-\eps_0,\eps_0] $ the function $u=\eps\hat u(\eps,\la)$
is a non-stationary solution  to \reff{problem}
with  $\om=\hat \om(\eps,\la)$ and  $\tau=\hat \tau(\eps,\la)$.

(ii)  {\rm Asymptotic expansion:} It holds
\beq
\label{uas}
[\hat u(0,0)](t,x)=\Re u_0(x)\cos t - \Im u_0(x)\sin t \mbox{ for all }
t \in \R \mbox{ and }x \in [0,1],
\ee
$\hat\om(0,0)=1$, $\hat\tau(0,0)=\tau_0$
and
\beq
\label{bifdir}
\d_\eps\hat\om(0,\la)=\d_\eps\hat\tau(0,\la)=0 \mbox{ for all } \la \in [-\eps_0,\eps_0].
\ee 

(iii) {\rm Local uniqueness:} There exists $\de>0$ such that for all 
solutions $(u,\om,\tau,\la)$ to \reff{problem} with $u \ne 0$ and
$\|u\|_\infty+ |\om-1| + |\tau-\tau_0|+|\la|<\de$
there exist $\eps\in (0,\eps_0]$ and $\vphi\in\R$ such that 
$\om=\hat\om(\eps,\la)$, $\tau=\hat\tau(\eps,\la)$ and $u(x,t)=\eps[\hat u(\eps,\la)](x,t+\vphi)$ for all  $t\in\R$
and $x\in[0,1]$.

(iv) {\rm Regularity:} For all $\eps\in [0,\eps_0]$, $\la \in [-\eps_0,\eps_0]$
and $k \in \N$ it holds $\hat u(\eps,\la) \in C^k_{2\pi}(\R \times [0,1])$.

(v) {\rm Smooth dependence:} The map $(\eps,\la) \in  [0,\eps_0]\times [-\eps_0,\eps_0]
\mapsto \hat u(\eps,\la) \in  C^k_{2\pi}(\R \times [0,1])$ is $C^\infty$-smooth for any $k \in \N$.
\end{thm}
\begin{rem}
  \label{param}
The parametrizations $u=\eps\hat{u}(\eps,\la)$, $\om=\hat{\om}(\eps,\la)$ and $\tau=\hat{\tau}(\eps,\la)$ depend on the choice of the
eigenfunctions $u_0$ and $u_*$, in general, while the sign of $\partial_\eps^2\hat{\tau}(0,0)$,
 determining the bifurcation direction, does not.
\end{rem}

 In descriptions of Hopf bifurcation phenomena one of the main questions is that of the so-called bifurcation direction,
 i.e. the question if the bifurcating time-periodic solutions exist for bifurcation parameters (close to the bifurcation point)
 such that the stationary solution is unstable (in this case the Hopf bifurcation is called supercritical) or not.
 For ODEs and parabolic PDEs
 (with or without delay) it is known that, under reasonable additional assumptions, in the supercritical case  the bifurcating
 time-periodic solutions are orbitally stable. For hyperbolic PDEs this relationship between bifurcation direction and stability
 is believed to be true also, but rigorous proofs are not available up to now. More exactly, it is expected that the bifurcating non-stationary time-periodic solutions,
which are described by Theorem \ref{thm:hopf}, are orbitally stable if for all eigenvalues $\mu\not=\pm i$ of \reff{evp} with $\tau=\tau_0$ it holds $\Re \mu<0$ and if
$$
\rho \d_\eps^2\hat{\tau}(0,0)>0.
$$

 Anyway, in Theorem \ref{thm:dir} below we present a formula which shows how to calculate the number $\partial_\eps^2\hat{\tau}(0,0)$ by
 means of the eigenfunctions $u_0$ and $u_*$ and and of the first three derivatives of the nonlinearity
 $b(x,0,\cdot,\cdot,\cdot,\cdot)$. It is known that those formulae may be quite complicated and not explicit
 (see, e.g., \cite[Section 3.3]{HWK}, \cite{Kayan}, \cite[Theorem I.12.2]{Ki}; \cite[Theorem 1.2(ii)]{KR3}, \cite{Li}).
Therefore, in order to keep the technicalities simple, in   Theorem \ref{thm:dir} below we consider only nonlinearities of the type
 \beq
 \label{rest}
 b(x,\la,u_1,u_2,u_3,u_4)=\sum_{j=1}^4\beta_j(x,\la,u_j) 
 \ee
 with $C^\infty$-functions $\beta_j:[0,1]\times\R^2 \to \R$ such that
 \beq
 \label{rest1}
 \beta_j(x,\la,0)= \partial^2_3\beta_j(x,0,0)=0 \mbox{ for all } j=1,2,3,4,\; x \in [0,1] \mbox{ and } \la \in \R.
 \ee
 Set
 $$
 \beta_j^0(x):=\partial_3^3\beta_j(x,0,0) \mbox{ for } j=1,2,3,4.
 $$

 Our result about the bifurcation direction reads as follows:
 \begin{thm}
\label{thm:dir}
Let the assumptions of Theorem \ref{thm:hopf} and the conditions
\reff{rest} and \reff{rest1} be fulfilled.
Then
$$
\partial_\eps^2\hat{\tau}(0,0)=\frac{3}{8\rho} \mbox{\rm Re}\left(\frac{1}{\sigma}\int_0^1\left(
    (\beta^0_1+\beta^0_2e^{-i\tau_0} +i\beta^0_3)|u_0|^2u_0+\beta^0_4
    |u'_0|^2u'_0\right)\overline{u_*}dx\right).
$$
\end{thm}
\begin{rem}
We do not know if generalizations of Theorems \ref{thm:hopf} and \ref{thm:dir}  to higher space dimensions and/or to
quasilinear equations exist and how they should look like.

Also, we do not know much about corresponding to \reff{eq:1.1} initial-boundary value problems.
Especially, we do not know if the bifurcation direction implies stability properties of the
bifurcating time-periodic solutions (as it is the case for ODEs or parabolic PDEs).
\end{rem}

Our paper is organized as follows:

In Subsection  \ref{related} we comment about some publications which are related to ours.

In Section \ref{sec:FOS} we show that any solution  to  \reff{problem} creates a solution to a semilinear first-order $2 \times 2$
hyperbolic system, namely \reff{FOS}, and vice versa.
In Section \ref{sec:Inteq} we show (by using the method of integration along characteristics)
that any solution to  the first-order 
hyperbolic system \reff{FOS} solves a system of partial integral equations, namely 
\reff{IE}, and vice versa. Remark that in Sections 
 \ref{sec:FOS} and \ref{sec:Inteq} we do pure transformations, i.e., problem \reff{problem} is equivalent to problem \reff{IE}.
Especially, the technical difficulties of \reff{problem}, like small divisors and loss of smoothness, are hidden in \reff{IE} also.
But it turns out that in \reff{IE} they can be handled more easily than in \reff{problem}.

In Sections \ref{Lyapunov-Schmidt Procedure} and \ref{sec:external}
we do a Lyapunov-Schmidt procedure in order to reduce locally
the system  \reff{IE} with infinite-dimensional state parameter 
to a problem with two-dimensional state parameter.
Here the main technical results are  Lemma \ref{Fredia} about Fredholmness of the linearization of \reff{IE}
and Lemma  \ref{4.10}
about local unique solvability and smooth dependence of the infinite dimensional part 
of the Lyapunov-Schmidt system.
The proofs of those lemmas are  much more complicated than the corresponding proofs for ODEs or parabolic PDEs 
(with or without delay). 

In particular, in the proof of Lemma \ref{Fredia} (more exactly in the proof of Claim 4 there)
we use assumption \reff{Fred}, and it turns out that the conclusions of Lemma \ref{Fredia} (and of Theorem \ref{thm:hopf} as well)
are not true, in general, if \reff{Fred} is not true. 

In the proof of Lemma \ref{4.10} we use a
generalized implicit function theorem, which is a particular case of \cite[Theorem 2.2]{KR4} and
concerns abstract parameter-dependent equations with a loss of smoothness property.
This generalized implicit function theorem  is presented in Subsection~\ref{appendix}.

In Section \ref{The bifurcation equation}
we put the solution of the  infinite dimensional part
of the Lyapunov-Schmidt system into the  finite dimensional part and discuss the behavior of the resulting
equation. This is completely analogous to what is known from Hopf bifurcation for ODEs and parabolic PDEs.

In Section \ref{sec:dir} we prove Theorem \ref{thm:dir} and give an example.

Finally, in Section \ref{other} we discuss cases of other than \reff{eq:1.2} boundary conditions.

\subsection{Remarks on related work}
\label{related}
The main methods for proving Hopf bifurcation theorems are, roughly speaking, center manifold reduction 
and Lyapunov-Schmidt reduction.
In order to apply them to evolution equations, one needs to have a smooth center manifold for 
the corresponding 
semiflow (for center manifold reduction) or a Fredholm property  of the linearized equation on 
spaces of periodic functions (for Lyapunov-Schmidt reduction). 

In  \cite{CrandRab,Ki} Hopf bifurcation theorems for abstract evolution equations are proved by means of
 Lyapunov-Schmidt reduction, and in \cite{Haragus,Ma,Vander} by means of center manifold reduction.
In  \cite{CrandRab,Ki} it is assumed that the operator of the linearized equation is sectorial
(see \cite[Hypothesis (HL)]{CrandRab} and \cite[Hypothesis I.8.8]{Ki}), hence this setting is not appropriate 
for hyperbolic PDEs. In  \cite{Haragus,Ma,Vander} the assumptions concerning the linearized operator are more 
general, including non-sectorial operators. However, it is unclear if our problem \reff{eq:1.1}, \reff{eq:1.2}
can be written as an abstract evolution equation satisfying those conditions.

In \cite{Vander} it is shown that 1D semilinear damped wave equations  without delay of the type
$\d_t^2u=\d_x^2u-\gamma \d_tu+f(u)$
with  $f(0)=0$, subjected to homogeneous Dirichlet boundary conditions, can be written as an abstract 
evolution equation satisfying the general assumptions of  \cite{Vander}, and a corresponding Hopf bifurcation theorem is proved.
But it turns out that nonlinearities of the type $f(u,\d_xu)$ cannot be treated this way.
In \cite{Koch} a Hopf bifurcation theorem is stated without proof for
 second-order quasilinear  hyperbolic systems without delay with arbitrary space dimension
subjected to homogeneous Dirichlet boundary conditions.
In \cite{KR3} a Hopf bifurcation theorem for general semilinear first-order 1D hyperbolic systems without delay
is proved by means of  Lyapunov-Schmidt reduction, and applications to semiconductor laser modeling are described.
In \cite{Liu,Magal,Magal1} the authors  considered Hopf bifurcation for
scalar linear first-order PDEs without delay of the type
$(\partial_t  +\partial_x + \mu)u = 0$ on the semi-axis $(0,\infty)$
with a nonlinear integral boundary condition at $x=0$.

What concerns Hopf bifurcation for hyperbolic PDEs with delay, to the best of our knowledge there exist only the two results
\cite{Kos1,Kos2} of N. Kosovali\'{c} and B. Pigott. In \cite{Kos1} the authors consider 1D damped and delayed
Sine-Gordon-like wave equations of the type
\beq
\label{Sine}
\d_t^2u(t,x)-\d_x^2u(t,x)+\d_tu(t,x)+u(t-\tau,x)=f(x,u(t-\tau,x))
\ee
with $f(-x,-u)=-f(x,u)$ and $f(x,0)=\partial_uf(x,0)=0$.
Because of the symmetry assumption on the nonlinearity $f$ the bifurcating time-periodic solutions can be determined
by means of Fourier expansions. In  \cite{Kos2} these results are generalized to equations on $d$-dimensional cubes,
but locally unique bifurcating solution families can be described for fixed prescribed spatial frequency vectors only.

Our results in the present paper extend those of  \cite{Kos1} mainly by two facts: Our equation \reff{eq:1.1} is
more general than \reff{Sine} (and does not have any symmetry property, in general), and we allow the presence of the perturbation parameter $\la$.
The symmetry assumptions of  \cite{Kos1} allow one to use Fourier series techniques, while 
we use integration along characteristics.

\section{Transformation of the second-order equation into a first-order system}
\label{sec:FOS}
\renewcommand{\theequation}{{\thesection}.\arabic{equation}}
\setcounter{equation}{0}

In this section we show that any solution $u$ to  \reff{problem} creates a solution $v=(v_1,v_2)$ to the first-order hyperbolic system
\beq
\label{FOS}
\left.
  \begin{array}{l}
    \om\d_tv_1(t,x)-a(x,\la)\d_xv_1(t,x)=[B(v,\om,\tau,\la)](t,x),\\
    \om\d_tv_2(t,x)+a(x,\la)\d_xv_2(t,x)=[B(v,\om,\tau,\la)](t,x),\\
    v_1(t,0)+v_2(t,0)= v_1(t,1)-v_2(t,1)=0,\\
    v(t+2\pi,x)=v(t,x)
  \end{array}
\right\}
\ee
and vice versa. Here the nonlinear operator $B$ is defined as
\begin{eqnarray}
\label{Bdef}
  [B(v,\om,\tau,\la)](t,x)&:=&b(x,\la,[J_\la v](t,x),[J_\la v](t-\om\tau,x), [Kv](t,x),[K_\la v](t,x))\nonumber\\
  &&  -\frac{1}{2}\d_xa(x,\la)(v_1(t,x)-v_2(t,x))
\end{eqnarray}
with partial integral operators $J_\la$ defined by
\beq
\label{Jdef}
[J_\la v](t,x):=\frac{1}{2}\int_0^x\frac{v_1(t,\xi)-v_2(t,\xi)}{a(\xi,\la)}d\xi
\ee
and with "pointwise" operators $K$ and $K_\la$ defined by 
\beq
\label{Kdef}
[Kv](t,x):=\frac{v_1(t,x)+v_2(t,x)}{2}, \; [K_\la v](t,x):=\frac{v_1(t,x)-v_2(t,x)}{2a(x,\la)}=[\d_xJ_\la v](t,x).
\ee
\begin{defn}\label{def:Cspaces2}
(i) We denote by $C_{2\pi}(\R\times [0,1];\R^2)$ the space of all continuous functions $v:[0,1] \times \R \to \R^2$ such that
$v(t+2\pi,x)=v(t,x)$ for all $t \in \R$ and $x \in [0,1]$, with the norm 
$$
\|v\|_\infty:=\max\{|v_1(t,x)|+|v_2(t,x)|: t \in \R, \; x \in [0,1]\}.
$$

(ii) For $k \in \N$  we denote by $C^k_{2\pi}(\R\times [0,1];\R^2)$ the space of all $C^k$-smooth functions
$v\in C_{2\pi}(\R\times [0,1];\R^2)$, with the norm 
$\max\{\|\d^i_t\d^j_xv\|_\infty: 0 \le i+j\le k\}$.
\end{defn}
\begin{lemma}\label{lem:FOS}
For all $\om,\tau,\la \in \R$ and $k=2,3,\ldots$ the following is true:

(i) If $u \in C_{2\pi}^k(\R\times [0,1])$ is a solution to \reff{problem}, then the function $v \in C_{2\pi}^{k-1}(\R\times [0,1];\R^2)$,
which is defined by
\beq
\label{vdef}
v_1(t,x):=\om \d_tu(t,x)+a(x,\la) \d_xu(t,x), \;v_2(t,x):=\om \d_tu(t,x)-a(x,\la) \d_xu(t,x),
\ee
is a solution to \reff{FOS}.

(ii) If  $v \in C_{2\pi}^{k-1}(\R\times [0,1];\R^2)$ is a solution to \reff{FOS}, then the function  $u \in C_{2\pi}^{k-1}(\R\times [0,1])$,
which is defined by
\beq
\label{udef}
u(t,x):=\frac{1}{2}\int_0^x\frac{v_1(t,\xi)-v_2(t,\xi)}{a(\xi,\la)}d\xi,
\ee
is $C^k$-smooth and a solution to \reff{problem}.
\end{lemma}
\begin{proof}
  (i) Let $u \in C_{2\pi}^k(\R\times [0,1])$ be a solution to \reff{problem}, and let $v \in C_{2\pi}^{k-1}(\R\times [0,1];\R^2)$
  be defined by \reff{vdef}. From  \reff{vdef} follows
  \begin{eqnarray*}
\d_tv_1=\om\d^2_tu+a\d_t\d_xu, & \d_xv_1=\om \d_t\d_xu+\d_xa \d_xu + a\d_x^2u,\\
    \d_tv_2=\om\d^2_tu-a\d_t\d_xu, & \d_xv_2=\om \d_t\d_xu-\d_xa \d_xu - a\d_x^2u.
  \end{eqnarray*}
  Hence
  \beq
  \label{KJ}
  \om\d_tu=\frac{v_1+v_2}{2}=Kv,\;  \d_xu=\frac{v_1-v_2}{2a}=K_\la v
  \ee
  and
   \beq
  \label{uv}
  \om^2\d_t^2u-a^2\d_x^2u-a\d_xa\d_xu=\om\d_tv_1-a\d_xv_1=\om\d_tv_2+a\d_xv_2.
  \ee
  From $u(t,0)=\d_xu(t,1)=0$ (cf. \reff{problem}) and \reff{KJ} follows $v_1(t,0)-v_2(t,0)=0$ and
  $v_1(t,1)+v_2(t,1)=0$, i.e. the boundary conditions of \reff{FOS}.
  Further, from $u(t,0)=0$ and \reff{KJ} follows also $u=J_\la v$.
  Hence, \reff{KJ}, \reff{uv} and the differential equation in \reff{problem} yield the differential equations in \reff{FOS}.

  (ii) Let $v \in C_{2\pi}^{k-1}(\R\times [0,1];\R^2)$ be a solution to \reff{FOS}, and let $u \in C_{2\pi}^{k-1}(\R\times [0,1])$
  be defined by \reff{udef}. From \reff{FOS} and \reff{udef} it follows that
  $$
  \d_tu(t,x)=\int_0^x\frac{\d_tv_1(t,\xi)-\d_tv_2(t,\xi)}{2a(\xi,\la)}d\xi=
  \int_0^x\frac{\d_xv_1(t,\xi)+\d_xv_2(t,\xi)}{2\om}d\xi=\frac{v_1(t,x)+v_2(t,x)}{2\om}.
  $$
  Hence, $\d_tu$ is $C^{k-1}$-smooth, and
  \beq
  \label{uv1}
  \om^2\d^2_tu=\frac{\om}{2}(\d_tv_1+\d_tv_2).
  \ee
  Further, \reff{udef} yields
   \beq
  \label{uv2}
  \d_xu=\frac{v_1-v_2}{2a}=K_\la v,
  \ee
  i.e.  $\d_xu$ is $C^{k-1}$-smooth also, i.e.  $u$ is $C^{k}$-smooth,
  and $2(\d_xa\d_xu+a\d_x^2u)=\d_xv_1-\d_xv_2$, i.e.
   \beq
  \label{uv3}
  a^2\d^2_xu=\frac{a}{2}(\d_xv_1-\d_xv_2)-\frac{\d_xa}{2}(v_1-v_2).
  \ee
  But \reff{FOS}, \reff{uv1} and \reff{uv3} imply
  $\om^2\d^2_tu-a^2\d^2_xu=B(v,\om,\tau,\la)+\frac{1}{2}\d_xa(v_1-v_2)$,
  i.e. the differential equation in \reff{problem}. The boundary conditions in \reff{problem} follow from the boundary conditions in \reff{FOS}
  and from \reff{udef} and \reff{uv2}.
\end{proof}

Let us calculate the linearization of the operator $B$ (cf. \reff{Bdef}) with respect to $v$ in $v=0$.
For that reason we use the following notation:
\beq
\label{bdef1}
b_j(x,\la):=\partial_jb(x,\la,0,0,0,0) \mbox{ for } j=3,4,5,6.
\ee
Remark that $b_j(x,0)=b_j^0(x)$ (cf. \reff{bdef}).
We have
\begin{eqnarray}
  \label{Bvdef}
  &&[\partial_vB(0,\om,\tau,\la)v](t,x)\nonumber\\
     &&=b_3(x,\la)[J_\la v](t,x)+b_4(x,\la)[J_\la v](t-\om \tau,x)+b_5(x,\la)[Kv](t,x)+
     b_6(x,\la)[K_\la v](t,x)\nonumber\\
     &&\;\;\;\;\;\;\;\;\;\;-\frac{1}{2}\d_xa(x,\la)(v_1(t,x)-v_2(t,x))\nonumber\\
  &&=b_1(x,\la)v_1(t,x)+b_2(x,\la)v_2(t,x)+b_3(x,\la)[J_\la v](t,x)+b_4(x,\la)[J_\la v](t-\om \tau,x)
\end{eqnarray}
with
\beq
\label{b1def}
\left.
  \begin{array}{rcl}
b_1(x,\la)&:=&\displaystyle\frac{1}{2}\left(-\partial_xa(x,\la)+b_5(x,\la)+\frac{b_6(x,\la)}{a(x,\la)}\right),\;\\
               b_2(x,\la)&:=&\displaystyle\frac{1}{2}\left(\partial_xa(x,\la)+b_5(x,\la)-\frac{b_6(x,\la)}{a(x,\la)}\right).
  \end{array}
  \right\}
  \ee
  By reasons which will be seen in Sections \ref{sec:Inteq} and \ref{Lyapunov-Schmidt Procedure} below we rewrite system \reff{FOS}
  in the following way:
\beq
\label{FOS1}
\left.
  \begin{array}{l}
    \om\d_tv_1(t,x)-a(x,\la)\d_xv_1(t,x)-b_1(x,\la)v_1(t,x)=[\B_1(v,\om,\tau,\la)](t,x),\\
    \om\d_tv_2(t,x)+a(x,\la)\d_xv_2(t,x)-b_2(x,\la)v_2(t,x)=[\B_2(v,\om,\tau,\la)](t,x),\\
    v_1(t,0)+v_2(t,0)= v_1(t,1)-v_2(t,1)=0,\\
    v(t+2\pi,x)=v(t,x)
  \end{array}
\right\}
\ee
with
\beq
\label{BBdef}
\left.
  \begin{array}{rcl}
    [\B_1(v,\om,\tau,\la)](t,x)&:=&[B(v,\om,\tau,\la)](t,x)-b_1(x,\la)v_1(t,x),\\
 \displaystyle[\B_2(v,\om,\tau,\la)](t,x)&:=&[B(v,\om,\tau,\la)](t,x)-b_2(x,\la)v_2(t,x).
\end{array}
  \right\}
  \ee
  The operators $\B_1,\B_2:C_{2\pi}(\R \times [0,1];\R^2)\times \R^3\to C_{2\pi}(\R \times [0,1])$,
  introduced in \reff{BBdef}, define an operator
  $$
  \B:=(\B_1,\B_2):C_{2\pi}(\R \times [0,1];\R^2)\times \R^3\to C_{2\pi}(\R \times [0,1];\R^2).
  $$
  Moreover, the operator $\B(\cdot,\om,\tau,\la):C_{2\pi}(\R \times [0,1];\R^2)\to C_{2\pi}(\R \times [0,1];\R^2)$
is $C^\infty$-smooth because the function $b$ is supposed to be  $C^\infty$-smooth, and
\beq
\label{diffB}
\partial_v\B(0,\om,\tau,\la)=\J(\om,\tau,\la) + \KK(\la)
\ee
with operators $\J(\om,\tau,\la),\KK(\la)\in {\cal L}(C_{2\pi}(\R \times [0,1];\R^2))$.
Their components
are defined by (cf. \reff{Bvdef} and \reff{BBdef})
\begin{eqnarray}
  &&[\J_1(\om,\tau,\la)v](t,x)=
     [\J_2(\om,\tau,\la)v](t,x):=b_3(x,\la) [J_\la v](t,x)+b_4(x,\la) [J_\la v](t-\om\tau,x)
     \nonumber\\
  &&= \displaystyle\frac{1}{2}\int_0^x\frac{b_3(x,\la)(v_1(t,\xi)-v_2(t,\xi))
     +b_4(x,\la)(v_1(t-\om\tau,\xi)-v_2(t-\om\tau,\xi))}{a(\xi,\la)}d\xi
     \label{J12def}
     \end{eqnarray}
and
\beq
\label{K12def}
 \displaystyle [\KK_1(\la)v](t,x)=b_2(x,\la)v_2(t,x),\;  \displaystyle [\KK_2(\la)v](t,x)=b_1(x,\la)v_1(t,x).
\ee
Hence, the linearization with respect to $v$ in $v=0$ of the right-hand side of \reff{FOS1} has a special structure: It is the sum
of the partial integral operator $\J(\om,\tau,\la)$ and of the ``pointwise'' operator $\KK(\la)$, which has vanishing diagonal part.
This structure will be used
in Subsection \ref{Fredholmness of the linearization} below, cf. Remark \ref{diago}.

\section{Transformation of the first-order system into a system of  partial integral equations}
\label{sec:Inteq}
\renewcommand{\theequation}{{\thesection}.\arabic{equation}}
\setcounter{equation}{0}

In this section we show (by using the method of integration along characteristics)
that any solution to \reff{FOS}, i.e. to \reff{FOS1}, solves the system of partial integral equations
\beq
\label{IE}
\left.
  \begin{array}{l}
    v_1(t,x)+c_1(x,0,\la)v_2(t+\om A(x,0,\la),0)\\
    =\displaystyle-\int_0^x\frac{c_1(x,\xi,\la)}
                {a(\xi,\la)}[\B_1(v,\om,\tau,\la)](t+\om A(x,\xi,\la),\xi)d\xi,\\
    v_2(t,x)-c_2(x,1,\la)v_1(t-\om A(x,1,\la),1)\\
    =\displaystyle\int_x^1\frac{c_2(x,\xi,\la)}
                {a(\xi,\la)}[\B_2(v,\om,\tau,\la)](t-\om A(x,\xi,\la),\xi)d\xi
  \end{array}
\right\}
\ee
and vice versa. Here the operators $\B_1$ and $\B_1$ are from \reff{BBdef}, and
the functions $c_1$, $c_2$ and $A$  are defined by (cf. \reff{bdef1} and \reff{b1def})
$$
  c_1(x,\xi,\la):=\exp\int_x^\xi\frac{b_1(\eta,\la)}{a(\eta,\la)}d\eta,\;
  c_2(x,\xi,\la):=\exp\int_\xi^x\frac{b_2(\eta,\la)}{a(\eta,\la)}d\eta,\; 
A(x,\xi,\la):=\int_\xi^x\frac{d\eta}{a(\eta,\la)}.
$$
                    
 \begin{lemma}\label{lem:parint}
For all $\om,\tau,\la \in \R$ the following is true:

(i) If $v \in C^1_{2\pi}(\R\times [0,1];\R^2)$ is a solution to \reff{FOS}, then it is a solution to \reff{IE}.

(ii) If  $v \in C_{2\pi}(\R\times [0,1];\R^2)$ is a solution to \reff{IE}
and if $\d_tv$ exists and is continuous, then $v$ belongs to  $C^1_{2\pi}(\R\times [0,1];\R^2)$
and solves \reff{FOS}.
\end{lemma}
\begin{proof}  (i) Let $v\in C^1_{2\pi}(\R\times [0,1];\R^2)$ be given.
  Because of $c_1(x,x,\la)=1$ and $A(x,x,\la)=0$ we get
  \begin{eqnarray*}
    && v_1(t,x)-c_1(x,0,\la)v_1(t+\om A(x,0,\la),0)=\int_0^x\frac{d}{d\xi}\left(c_1(x,\xi,\la)v_1(t+\om A(x,\xi,\la),\xi)\right)d\xi\\
    &&=\int_0^x\partial_\xi c_1(c,\xi,\la)v_1(t+\om A(x,\xi,\la),\xi)d\xi\\
    &&\;\;\;\;\;+\int_0^xc_1(x,\xi,\la)\left(\partial_tv_1(t+\om A(x,\xi,\la),\xi)\om \partial_\xi A(x,\xi,\la)
       +\partial_xv_1(t+\om A(x,\xi,\la),\xi)\right)d\xi.
  \end{eqnarray*}
  From $\partial_\xi A(x,\xi,\la)=-1/a(\xi,\la)$ and $\partial_\xi c_1(x,\xi,\la)=b_1(\xi,\la)c_1(x,\xi,\la)/a(\xi,\la)$ it follows that
  \begin{eqnarray*}
    && v_1(t,x)-c_1(x,0,\la)v_1(t+\om A(x,0,\la),0)\nonumber\\
    &&=\int_0^x\frac{c_1(x,\xi,\la)}{a(\xi,\la)}\Big[-\om \partial_tv_1(s,\xi)
    +a(\xi,\la)\partial_xv_1(s,\xi)+b_1(\xi,\la)v_1(s,\xi)\Big]_{s=t+\om A(x,\xi,\la)}d\xi.
  \end{eqnarray*}
Similarly one shows that
 \begin{eqnarray*}
    && v_2(t,x)-c_2(x,1,\la)v_1(t-\om A(x,1,\la),1)=-\int_x^1\frac{d}{d\xi}\left(c_2(x,\xi,\la)v_2(t-\om A(x,\xi,\la),\xi)\right)d\xi\\
    &&=-\int_x^1\frac{c_2(x,\xi,\la)}{a(\xi,\la)}\Big[\om \partial_tv_1(s,\xi)
    +a(\xi,\la)\partial_xv_1(s,\xi)-b_2(\xi,\la)v_1(s,\xi)\Big]_{s=t-\om A(x,\xi,\la)}d\xi.
  \end{eqnarray*}
But this yields that, if $v$ is a solution to \reff{FOS}, i.e. to \reff{FOS1}, then it is a solution to \reff{IE}

  (ii) Let  $v \in C_{2\pi}(\R\times [0,1];\R^2)$ be a solution to \reff{IE}.
  The first equation of \reff{IE} yields $v_1(t,0)=c_1(0,0,\la)v_2(t+\om A(0,0,\la),0)=-v_2(t,0)$, i.e. the first boundary condition of
  \reff{FOS}. Similarly the second boundary condition of
  \reff{FOS} follows from the  second equation of \reff{IE}.

  Further, from  \reff{IE} and from the assumption, that
  $\d_tv$ exists and is continuous, it follows that also  $\d_xv$ exists and is continuous, i.e. $v\in C^1_{2\pi}(\R\times [0,1];\R^2)$.

  Now, let us verify the differential equations in \reff{FOS}, i.e. in \reff{FOS1}. From \reff{IE} it follows that
  \begin{eqnarray}
    \label{rhs}
    &&\left(\om\d_t-a(x,\la)\d_x\right)\left(v_1(t,x)+c_1(x,0,\la)v_2(t+\om A(x,0,\la),0)\right)\nonumber\\
    &&=-\left(\om\d_t-a(x,\la)\d_x\right)\int_0^x\frac{c_1(x,\xi,\la)}{a(\xi,\la)}
       [\B_1(v,\om,\tau,\la)](t+\om A(x,\xi,\la),\xi)d\xi.
       \end{eqnarray}
Because of $\partial_xc_1(x,0,\la)=-b_1(x,\la)c_1(x,0,\la)/a(x,\la)$ and 
 $$
  \left(\om\d_t-a(x,\la)\d_x\right)\vp(t+\om A(x,\xi,\la))=0 \mbox{ for all } \vp \in C^1(\R)
  $$
the left-hand side of \reff{rhs} is
$
\left(\om\d_t-a(x,\la)\d_x\right)v_1(t,x)+b_1(x,\la)v_2(t+\om A(x,0,\la),0),
$
and the  right-hand side of \reff{rhs} is
$$
b_1(x,\la)\int_0^x\frac{c_1(x,\xi,\la)}{a(\xi,\la)}
       [\B_1(v,\om,\tau,\la)](t+\om A(x,\xi,\la),\xi)d\xi+[\B_1(v,\om,\tau,\la)](t,x).
       $$
       Hence, the first equation of \reff{FOS1} is shown.
       Using  $\partial_xc_2(x,0,\la)=b_2(x,\la)c_1(x,0,\la)/a(x,\la)$, one gets  similarly
       \begin{eqnarray*}
    &&\left(\om\d_t+a(x,\la)\d_x\right)\left(v_2(t,x)+c_2(x,0,\la)v_1(t-\om A(x,0,\la),0)\right)\\
    &&=\left(\om\d_t+a(x,\la)\d_x\right)v_2(t,x) +b_2(x,\la)v_1(t-\om A(x,0,\la),0)\\
    &&=- \left(\om\d_t+a(x,\la)\d_x\right)\int_x^1\frac{c_2(x,\xi,\la)}{a(\xi,\la)}
       [\B_2(v,\om,\tau,\la)](t-\om A(x,\xi,\la),\xi)d\xi\\
    &&=b_2(x,\la)\int_x^1\frac{c_2(x,\xi,\la)}{a(\xi,\la)}
       [\B_2(v,\om,\tau,\la)](t+\om A(x,\xi,\la),\xi)d\xi+[\B_2(v,\om,\tau,\la)](t,x),
  \end{eqnarray*}
  i.e. the second equation of \reff{FOS1} is shown.
\end{proof}

\section{Lyapunov-Schmidt procedure}
\label{Lyapunov-Schmidt Procedure}
\renewcommand{\theequation}{{\thesection}.\arabic{equation}}
\setcounter{equation}{0}

In this section we do a Lyapunov-Schmidt procedure in order to reduce locally for $v \approx 0$, $\om \approx 1$, $\tau \approx \tau_0$ and
  $\la \approx 0$
the problem \reff{IE} with infinite-dimensional state parameter $(v,\om)$ to a problem with a two-dimensional state parameter.

For the sake of simplicity, we will write the problem \reff{IE} in a more abstract way. For that reason for $\om,\la \in \R$ let us introduce
operators
$
\CC(\om,\la),\DD(\om,\la)\in {\cal L}(C_{2\pi}(\R \times [0,1];\R^2))
$
with components  $\CC_j(\om,\la),\DD_j(\om,\la)\in {\cal L}(C_{2\pi}(\R \times [0,1]);\R^2);C_{2\pi}(\R \times [0,1]))$, $j=1,2$, which are defined by
\beq
\label{Cdef}
\left.
\begin{array}{rcl}  
[\CC_1(\om,\la)v](x,t)&:=&-c_1(x,0,\la)v_2(t+\om A(x,0,\la),0),\\
\displaystyle[\CC_2(\om,\la)v](x,t)&:=&c_2(x,1,\la)v_1(t-\om A(x,1,\la),1)
\end{array}
\right\}
\ee
and 
\beq
\label{Ddef}
\left.
\begin{array}{rcl}  
[\DD_1(\om,\la)v](x,t)&:=& \displaystyle
    -\int_0^x\frac{c_1(x,\xi,\la)}{a(\xi,\la)}v_1(t+\om A(x,\xi,\la),\xi)d\xi,\\
\displaystyle[\DD_2(\om,\la)v](x,t)&:=&\displaystyle
\int_x^1\frac{c_2(x,\xi,\la)}{a(\xi,\la)}v_2(t-\om A(x,\xi,\la),\xi)d\xi.
\end{array}
\right\}
\ee
Using this notation, the system \reff{IE} reads
\beq
\label{abstract}
v=\CC(\om,\la)v+\DD(\om,\la)\B(v,\om,\tau,\la),
\ee
where the nonlinear operator $\B$ is introduced in \reff{BBdef}.
\begin{rem}
  Also the first-order hyperbolic system \reff{FOS1} can be written in an abstract way, namely as
  \beq
  \label{FOSab}
  \A(\om,\la)v=\B(v,\om,\tau,\la)
  \ee
  with $\A(\om,\la) \in {\cal L}(C^1_{2\pi}(\R\times[0,1];\R^2);C_{2\pi}(\R\times[0,1];\R^2))$ defined by
  \beq
  \label{AAdef}
  [\A(\om,\la)v](t,x):=
  \left[
    \begin{array}{c}
      \om\d_tv_1(t,x)-a(x,\la)\d_xv_1(t,x)-b_1(x,\la)v_1(t,x)\\
       \om\d_tv_2(t,x)+a(x,\la)\d_xv_2(t,x)-b_2(x,\la)v_2(t,x)
    \end{array}
  \right].
  \ee
  Remark that in the proof of Lemma \ref{lem:parint} we showed that for all $\om,\la \in \R$ it holds
  \beq
  \label{Aid1}
  \A(\om,\la)\CC(\om,\la)v= \A(\om,\la)\DD(\om,\la)v-v=0 \mbox{ for all } v \in C^1_{2\pi}(\R\times[0,1];\R^2)
  \ee
  and
\beq
  \label{Aid2}
  \DD(\om,\la)\A(\om,\la)v=v-\CC(\om,\la)v \mbox{ for all } v \in C^1_{2\pi}(\R\times[0,1];\R^2)
  \mbox{ with } [v_1+v_2]_{x=0}=[v_1-v_2]_{x=1}=0.
  \ee
\end{rem}

It is easy to see that the operators  $\CC(\om,\la),\DD(\om,\la)$, $\J(\om,\tau,\la)$ and $\KK(\la)$
(cf. \reff{J12def}, \reff{K12def}) are bounded with respect to $\om$ and $\tau$ and locally bounded with respect to $\la$, i.e., for any $c>0$ it holds
   \beq
   \label{bounded}
   \sup_{\om,\tau \in \R, |\la|\le c}\{\|\CC(\om,\la)v\|_\infty+\|\DD(\om,\la)v\|_\infty+\|\J(\om,\tau,\la)v\|_\infty
   +\|\KK(\la)v\|_\infty :\; \|v\|_\infty \le 1\}<\infty.
   \ee
But, unfortunately, the operators $\CC(\om,\la)$ and $\DD(\om,\la)$ 
do not depend continuously (in the sense of the uniform operator norm in
${\cal L}(C_{2\pi}(\R \times [0,1];\R^2))$ on $\om$ and $\la$, in general, and  $\J(\om,\tau,\la)$
does not depend  continuously on $\om$ and $\tau$, in general.
This is the main technical difficulty which we have to overcome
in order to analyze the bifurcation problem \reff{abstract}.
Remark that this difficulty appears also in the case if $\tau$
is fixed to be zero (and $\la$ is used to be the bifurcation parameter),
i.e. in the case of Hopf bifurcation for semilinear wave equations without delay.

\subsection{Fredholmness of the linearization}
\label{Fredholmness of the linearization}
We intend to show that the linearization  of \reff{abstract} at $v=0$, i.e., the operator
$$
I-\CC(\om,\la)-\DD(\om,\la)\partial_v\B(0,\om,\tau,\la)=I-\CC(\om,\la)-\DD(\om,\la)(\J(\om,\tau,\la)+\KK(\la)),
$$
is a Fredholm operator of index zero from
the space $C_{2\pi}(\R \times [0,1];\R^2)$
into itself. 
\begin{lemma}\label{Fredia}
  Let the condition \reff{Fred} be fulfilled. Then there exists $\delta>0$ such that for all $\om,\tau,\la \in \R$ with  $\om \ne 0$
  and $|\la|<\delta$
  the operator $I-\CC(\om,\la)-\DD(\om,\la)\partial_v\B(0,\om,\tau,\la)$ is a Fredholm operator of index zero from
  the space $C_{2\pi}(\R \times [0,1];\R^2)$ into itself.
\end{lemma}
The main complication in the proof is caused by the fact that
 the operators $\CC(\om,\la)+\DD(\om,\la)\partial_v\B(0,\om,\tau,\la)$ are not completely continuous from
the space $C_{2\pi}(\R \times [0,1];\R^2)$ into itself, in general. 

The proof will be divided into a number of claims.

\begin{claim}\label{Jlemm}
  For all $\om,\tau,\la \in \R$ with  $\om \ne 0$ and all $v \in C_{2\pi}(\R \times [0,1];\R^2)$
  we have $\DD(\om,\la)\J(\om,\tau,\la)v \in C^1_{2\pi}(\R \times [0,1];\R^2)$, and for any $c>0$ it holds
  \beq
\label{sup1}
  \sup_{1/c \le \om \le c, \tau \in \R, |\la| \le c}\{\|\partial_t\DD(\om,\la)\J(\om,\tau,\la)v\|_\infty+\|\partial_x\DD(\om,\la)\J(\om,\tau,\la)v\|_\infty:
  \;\|v\|_\infty \le 1\}<\infty.
  \ee
\end{claim}
\begin{subproof}
	The idea of the proof is to show that the composition of the two
	partial integral operators  $\DD(\om,\la)$ and  $\J(\om,\tau,\la)$
  is an  integral operator mapping  $C_{2\pi}(\R \times [0,1];\R^2)$
  into  $C^1_{2\pi}(\R \times [0,1];\R^2)$.
  Indeed, for $v \in C_{2\pi}(\R \times [0,1];\R^2)$ we have
  \begin{eqnarray*}
    &&[\DD_1(\om,\la)\J(\om,\tau,\la)v](t,x)\\
    &&=-\frac{1}{2}\int_0^x\left(\frac{c_1(x,\xi,\la)b_3(\xi,\la)}{a(\xi,\la)}\int_0^\xi
       \frac{v_1(t+\om A(\eta,\xi,\la),\eta)-v_2(t+\om A(\eta,\xi,\la),\eta)}{a(\eta,\la)}d\eta\right)d\xi\\
     &&-\frac{1}{2}\int_0^x\left(\frac{c_1(x,\xi,\la)b_4(\xi,\la)}{a(\xi,\la)}\int_0^\xi
      \frac{v_1(t-\om\tau+\om A(\eta,\xi,\la),\eta)-v_2(t-\om\tau+\om A(\eta,\xi,\la),\eta)}{a(\eta,\la)}d\eta\right)d\xi,
  \end{eqnarray*}
  where
  \begin{eqnarray}
    &&\int_0^x\int_0^\xi\frac{c_1(x,\xi,\la)b_3(\xi,\la)v_1(t+\om A(\eta,\xi,\la),\eta)}{a(\xi,\la)a(\eta,\la)}\,d\eta d\xi\nonumber\\
    &&=\int_0^x\int_\eta^x\frac{c_1(x,\xi,\la)b_3(\xi,\la)v_1(t+\om A(\eta,\xi,\la),\eta)}{a(\xi,\la)a(\eta,\la)}\,d\xi d\eta\nonumber\\
    &&=-\frac{1}{\om}\int_0^x\int_t^{t+\om A(\eta,x,\la)}\frac{c_1(x,\xi_{\eta,t,\om,\la}(\zeta),\la)b_3(\xi_{\eta,t,\om,\la}(\zeta),\la)
       v_1(\zeta,\eta)}{a(\eta,\la)}\,d\zeta d\eta.
        \label{change}
 \end{eqnarray}
 Here we changed the integration variable $\xi$ to a new integration variable
$$
\zeta=\zeta_{\eta,t,\om,\la}(\xi):= t+\om A(\eta,\xi,\la)=t+\om\int_\xi^\eta\frac{dz}{a(z,\la)},\;
d\zeta=-\frac{\om}{a(\xi,\la)}\,d\xi.
$$
Note that for   $\om \ne 0$ the inverse transformation $\xi=\xi_{\eta,t,\om,\la}(\zeta)$ exists and depends smoothly on
$\eta,t,\om,\la$ and $\zeta$.

Obviously, the absolute values of the partial derivatives of  \reff{change} with respect to $t$ and $x$ 
exist and 
can be estimated from above by
a constant times $\|v\|_\infty$.
Moreover, as long as $\om$ and $\la$ are varying in the ranges
$1/c\le\om \le c$ and $|\la| \le c$,
the constant may be chosen to be independent on 
$\om,\tau$ and $\la$ (and to depend on $c$ only).
The same
can be shown 
for the terms
$$
\int_0^x\int_0^\xi\frac{c_1(x,\xi,\la)b_3(\xi,\la)v_2(t+\om A(\eta,\xi,\la),\eta)}{a(\xi,\la)a(\eta,\la)}\,d\eta d\xi
$$
and
$$
\int_0^x\int_0^\xi\frac{c_1(x,\xi,\la)b_4(\xi,\la)v_j(t-\om\tau+\om A(\eta,\xi,\la),\eta)}{a(\xi,\la)a(\eta,\la)}\,d\eta d\xi,
j=1,2.
$$
Claim \ref{Jlemm} is therefore  proved for the first component
$\DD_1(\om,\la)\J(\om,\tau,\la)$.
The same argument applies to  the second component $\DD_2(\om,\la)\J(\om,\tau,\la)$.
\end{subproof}
\begin{claim}\label{DKDlemm}
  For all $\om,\tau,\la \in \R$ with  $\om \ne 0$ and all $v \in C_{2\pi}(\R \times [0,1];\R^2)$
  we have $\DD(\om,\la)\KK(\la)\DD(\om,\la)v \in C^1_{2\pi}(\R \times [0,1];\R^2)$, and for any $c>0$ it holds
  \beq
\label{sup2}
  \sup_{1/c \le \om \le c, \tau \in \R, |\la| \le c}\{\|\partial_t\DD(\om,\la)\KK(\la)\DD(\om,\la)v\|_\infty+\|\partial_x\DD(\om,\la)\KK(\la)\DD(\om,\la)v\|_\infty:
  \;\|v\|_\infty \le 1\}<\infty.
  \ee
\end{claim}
\begin{subproof}
  The proof is similar to the proof of Claim \ref{Jlemm}.
  We have
  \begin{eqnarray*}
    &&[\DD_1(\om,\la)\KK(\la)\DD(\om,\la)v](t,x)\\
    &&=
    -\int_0^x\int_0^\xi\frac{c_1(x,\xi,\la)c_2(\xi,\eta,\la)b_2(\xi,\la)}{a(\xi,\la)a(\eta,\la)}
       v_2(t+\om A(x,\xi,\la)-\om A(\xi,\eta,\la),\eta)\,d\eta d\xi\\
  &&=
  -\int_0^x\int_\eta^x\frac{c_1(x,\xi,\la)c_2(\xi,\eta,\la)b_2(\xi,\la)}{a(\xi,\la)a(\eta,\la)}
     v_2(t+\om A(x,\xi,\la)-\om A(\xi,\eta,\la),\eta)\,d\xi d\eta\\
    &&  =
    \frac{1}{2\om}\int_0^x\int_{t+\om A(x,\eta,\la)}^{t-\om A(x,\eta,\la)}
       \frac{c_1(x,\xi_{\eta,t,\om,\la}(\zeta),\la)c_2(\xi_{\eta,t,\om,\la}(\zeta),\eta,\la)
b_1(\xi_{\eta,t,\om,\la}(\zeta),\la)}{a(\eta,\la)}
       v_2(\zeta,\eta)\,d\zeta d\eta.
\end{eqnarray*}
Here we changed the integration variable $\xi$ to 
$$
\zeta=\zeta_{\eta,t,\om,\la}(\xi):= t+\om(A(x,\xi,\la)-A(\xi,\eta,\la))=t+\om\left(\int_\xi^x\frac{dz}{a(z,\la)}+
  \int_\xi^\eta\frac{dz}{a(z,\la)}\right),\;
d\zeta=-\frac{2\om}{a(\xi,\la)}d\xi,
$$
and denoted by $\xi=\xi_{\eta,t,\om,\la}(\zeta)$  the inverse transformation.
Now we proceed as in the proof of Claim \ref{DKDlemm}.
\end{subproof}
\begin{rem}
  \label{diago}
  In the proof of Claim \ref{DKDlemm} we used that the diagonal part of the operator $\KK(\la)$ vanishes.
  Indeed, if in place of \reff{K12def} we would have, for example,
  $
  [\KK_1(\la)v](t,x)=v_1(t,x)+b_2(x,\la)v_2(t,x),
    $
    then in $[\DD_1(\om,\la)\KK(\la)\DD(\om,\la)v](t,x)$ there would appear the additional summand
    $$
  -\int_0^x\int_\eta^x\frac{c_1(x,\xi,\la)c_2(\xi,\eta,\la)}{a(\xi,\la)a(\eta,\la)}
  v_1(t+\om A(x,\xi,\la)+\om A(\xi,\eta,\la),\eta)\,d\xi d\eta.
  $$
  Because of $A(x,\xi,\la)+A(\xi,\eta,\la)=A(x,\eta,\la)$ this equals to
  $$
 -\int_0^x\int_\eta^x\frac{c_1(x,\xi,\la)c_2(\xi,\eta,\la)}{a(\xi,\la)a(\eta,\la)}
   v_1(t+\om A(x,\eta,\la),\eta)\,d\xi d\eta,
   $$
   and this is not differentiable with respect to $t$, in general, if $v_1$  is not differentiable with respect to $t$.
 \end{rem}
\begin{claim}\label{DKClemm}
  For all $\om,\tau,\la \in \R$ with  $\om \ne 0$ and all $v \in C_{2\pi}(\R \times [0,1];\R^2)$
  we have $\DD(\om,\la)\KK(\la)\CC(\om,\la)v \in C^1_{2\pi}(\R \times [0,1];\R^2)$, and for any $c>0$ it holds
  \beq
\label{sup3}
  \sup_{1/c \le \om \le c, \tau \in \R, |\la| \le c}\{\|\partial_t\DD(\om,\la)\KK(\la)\CC(\om,\la)v\|_\infty+\|\partial_x\DD(\om,\la)\KK(\la)\CC(\om,\la)v\|_\infty:
  \;\|v\|_\infty \le 1\}<\infty.
  \ee
\end{claim}
\begin{subproof}
  We have
  \begin{eqnarray*}
    &&[\DD_1(\om,\la)\KK(\la)\CC(\om,\la)v](t,x)\\
    &&=-
    \int_0^x\frac{c_1(x,\xi,\la)c_2(\xi,1,\la)b_1(\xi,\la)}{a(\xi,\la)}
       v_1(t+\om A(x,\xi,\la)-\om A(\xi,1,\la),1)d\xi\\
    &&=\frac{1}{2\om}\int_{t+\om A(x,0,\la)-\om A(0,1,\la)}^{t-\om A(x,1,\la)}c_1(x,\xi_{t,\om,\la}(\zeta),\la)c_2(\xi_{t,\om,\la}(\zeta),1,\la)
       b_1(\xi_{t,\om,\la}(\zeta),\la)
       v_1(\zeta,1)d\zeta.
\end{eqnarray*}
Here we changed the integration variable $\xi$ to 
$$
\zeta=\zeta_{t,\om,\la}(\xi):= t+\om A(x,\xi,\la)-\om A(\xi,1,\la)=t+\om\int_\xi^x\frac{dz}{a(z,\la)}+\om\int_\xi^1\frac{dz}{a(z,\la)},\;
d\zeta=-\frac{2\om}{a(\xi,\la)}d\xi,
$$
and $\xi=\xi_{t,\om,\la}(\zeta)$ is the inverse transformation.
Again, now we can proceed as in the proof of Claim \ref{Jlemm}.
\end{subproof}
\begin{claim}\label{C1}
Let the condition \reff{Fred}  be fulfilled. Then there exists $\delta>0$ 
such that for all $\om,\la \in \R$ with $|\la|\le \delta$
the operator  $I-\CC(\om,\la)$ is an  isomorphism from $C_{2\pi}(\R \times [0,1];\R^2)$ to itself. Moreover, 
\beq
\label{I-C-1}
\sup_{\om \in \R, |\la|\le \delta} \left\{\|(I-\CC(\om,\la))^{-1}f\|_\infty:\; \|f\|_\infty \le 1\right\}<\infty. 
\ee
\end{claim}
\begin{subproof} Take $f \in C_{2\pi}(\R \times [0,1];\R^2)$. We have to show that  for all real numbers $\om$ and $\la$  with $\la \approx 0$
  there exists a unique function  $v \in C_{2\pi}(\R \times [0,1];\R^2)$ 
  satisfying the equation
  \beq
  \label{1a}
  (I-\CC(\om,\la))v=f
  \ee
  and that $\|v\|_\infty \le \mbox{const}\|f\|_\infty$, where the constant does not depend on $\om,\la$ and $f$.
 Equation \reff{1a} is satisfied if and only if for all $t \in \R$ and $x \in [0,1]$
  it holds
  \begin{eqnarray}
    \label{2a}
    v_1(t,x)&=&-c_1(x,0,\la)v_2(t+\om A(x,0,\la),0)+f_1(t,x),\\
     \label{3a}
    v_2(t,x)&=&c_2(x,0,\la)v_1(t-\om A(x,1,\la),1)+f_2(t,x).
    \end{eqnarray}
 System \reff{2a}, \reff{3a} is satisfied if and only if \reff{2a} is true and
 if it holds
  \begin{eqnarray}
     \label{4a}
    v_2(t,x)&=&c_2(x,1,\la)(-c_1(1,0,\la)v_2(t+\om (A(1,0,\la)-A(x,1,\la)),0)+f_1(t-\om A(x,1,\la),x))\nonumber\\
    &&+f_2(t,x),
  \end{eqnarray}
  i.e., if and only if  \reff{2a} and \reff{4a} are true and if
   \begin{eqnarray}
     \label{5a}
    v_2(t,0)&=&c_2(0,1,\la)(-c_1(1,0,\la)v_2(t+\om (A(1,0,\la)-A(0,1,\la)),0)+f_1(t-\om A(0,1,\la),0))\nonumber\\
    &&+f_2(t,0).
   \end{eqnarray}

   Equation \reff{5a} is a functional equation for the unknown function $v_2(\cdot,0)$. In order to solve this equation let us denote
   by $C_{2\pi}(\R)$ the Banach space of all $2\pi$-periodic continuous functions $\tilde{v}:\R \to\R$ with the norm
   $\|\tilde{v}\|_\infty:=\max\{|\tilde{v}(t)|: \; t \in \R\}$.  Equation \reff{5a} is an equation in  $C_{2\pi}(\R)$ of the type
   \beq
   \label{6a}
   (I-\tilde{\CC}(\om,\la))\tilde{v}=\tilde{f}(\om,\la)
   \ee
   with $\tilde{v},\tilde{f}\in C_{2\pi}(\R)$ defined by $\tilde{v}(t):=v_2(t,0)$ and
   \beq
   \label{fdef}
   [\tilde{f}(\om,\la)](t):=c_2(0,1,\la)f_1(t-\om A(0,1,\la),x)+f_2(t,0)
   \ee
   and with $\tilde{\CC}(\om,\la)\in {\cal L}(C_{2\pi}(\R))$ defined by
   \beq
   \label{tildedef}
   [\tilde{\CC}(\om,\la)\tilde{v}](t):=-c_1(1,0,\la)c_2(0,1,\la)\tilde{v}(t+\om (A(1,0,\la)-A(0,1,\la))).
   \ee
   From the definitions of the functions $c_1$ and $c_2$ it follows that
   $$
   c_1(1,0,\la)c_2(0,1,\la)=\exp \int_0^1\frac{b_5(x,\la)}{a(x,\la)}dx,
   $$
   and assumption \reff{Fred} yields
   $$
   c_0:= c_1(1,0,0)c_2(0,1,0)\not=1.
   $$
   Now, we distinguish two cases.

   {\it Case 1:
     $c_0<1$.} Then there exists $\delta>0$ such that for all  $\la \in [-\delta,\delta]$
   it holds
   $c_1(1,0,\la)c_2(0,1,\la)\le \frac{1+c_0}{2}<1$.
   Therefore
   $$
   \|\tilde{\CC}(\om,\la)\|_{ {\cal L}(C_{2\pi}(\R))}\le \frac{1+c_0}{2}<1 \mbox{ for all } \la \in [-\delta,\delta].
   $$
   Hence,  for all  $\la \in [-\delta,\delta]$ the operator $I-\tilde{\CC}(\om,\la)$ is an isomorphism from $C_{2\pi}(\R)$ to itself, and
   $$
   \|(I-\tilde{\CC}(\om,\la))^{-1}\|_{ {\cal L}(C_{2\pi}(\R))}\le \frac{1}{1-\frac{1+c_0}{2}}=\frac{2}{1-c_0}.
   $$
   Therefore, for all $\om,\la \in \R$ with $|\la| \le \delta$ there exists exactly one solution $v_2(\cdot,0)\in C_{2\pi}(\R)$
   to \reff{5a}, and
   $$
   \|v_2(\cdot,0)\|_\infty \le \mbox{const}\|\tilde{f}(\om,\la)\|_\infty\le \mbox{const}\|f\|_\infty,
   $$
   where the constants do not depend on $\om,\la$ and $f$.
   Inserting this solution into the right-hand side of \reff{4a} we get $v_2 \in  C_{2\pi}(\R\times [0,1])$,
   and  inserting this  into the right-hand side of \reff{2a} we get finally $v_1 \in  C_{2\pi}(\R\times [0,1])$,
   i.e. the unique solution $v=(v_1,v_2) \in  C_{2\pi}(\R\times [0,1];\R^2)$ to \reff{2a}, \reff{3a} such that
   $\|v\|_\infty \le \mbox{const}\|f\|_\infty$,   where the constant does not depend on $\om,\la$ and $f$.

   {\it Case 2: $c_0>1.$} 
Then there exists $\delta>0$ such that for all  $\la \in [-\delta,\delta]$
   it holds
   $c_1(1,0,\la)c_2(0,1,\la)\ge \frac{1+c_0}{2}>1$.
   Equation \reff{5a} is equivalent to
 \begin{eqnarray*}
   v_2(t,0)&=&\frac{v_2(t+\om (A(0,1,\la)-A(1,0,\la)),0)}{c_1(1,0,\la)c_2(0,1,\la)}\nonumber\\
   &&-\frac{f_1(t-\om A(1,0,\la),1))}{c_1(1,0,\la)}-\frac{f_2(t+\om(A(0,1,\la)-A(1,0,\la)),0)}{c_1(1,0,\la)c_2(0,1,\la)}.
 \end{eqnarray*}
 This  equation is again of the type \reff{6a}, but now with $\|\widetilde{\CC}(1,0)\|_{ {\cal L}(C_{2\pi}(\R))}\le1/c_0$.
 Hence, there exists $\delta>0$ such that
   $$
   \|\widetilde{\CC}(\om,\la)\|_{ {\cal L}(C_{2\pi}(\R))}\le \frac{2}{1+c_0}<1 \mbox{ for all } \la \in [-\delta,\delta].
   $$
   we can, therefore, proceed as in the case  $c_0<1$.
 \end{subproof}
 \begin{rem}
   \label{c1}
   Definition \reff{tildedef} implies that
   $\frac{d}{dt}\widetilde{\CC}(\om,\la)\tilde{v}=\widetilde{\CC}(\om,\la)\frac{d}{dt}\tilde{v}$ for all $\tilde{v}\in C^1_{2\pi}(\R)$. This yields the estimate
   $$
   \left\|\widetilde{\CC}(\om,\la)\tilde{v}\right\|_\infty+\left\|\frac{d}{dt}\widetilde{\CC}(\om,\la)\tilde{v}\right\|_\infty
   \le\left\|\widetilde{\CC}(\om,\la)\right\|_{{\cal L}( C_{2\pi}(\R))}(\|\tilde{v}\|_\infty+\|\tilde{v}'\|_\infty)
   \mbox{ for all } \tilde{v} \in C^1_{2\pi}(\R).
   $$
   Hence, $(I-\widetilde{\CC}(\om,\la))^{-1}$ is a linear bounded operator from $C^1_{2\pi}(\R)$ into   $C^1_{2\pi}(\R)$ for $\la \approx 0$.
   It follows that, for given $f \in C^1_{2\pi}(\R \times [0,1];\R^2)$, the solution $v$ to \reff{1a}  belongs to $C^1_{2\pi}(\R \times [0,1];\R^2)$
   and, moreover,
   \beq
   \label{Cdelta}
   \sup_{\om \in \R, |\la| \le \delta}\|(I-\CC(\om,\la))^{-1}\|_{{\cal L}(C^1_{2\pi}(\R \times [0,1];\R^2))}<\infty.
   \ee
   \end{rem}

Let us turn back to Fredholmness of the operator 
$I-\CC(\om,\la)-\DD(\om,\la)(\J(\om,\tau,\la)+\KK(\la))$
 from the space
 $C_{2\pi}(\R \times [0,1];\R^2)$ into itself for $\om \ne 0$ and $\la \approx 0$.
 Note that the space  $C^1_{2\pi}(\R \times [0,1];\R^2)$ is completely continuously embedded into the space
 $C_{2\pi}(\R \times [0,1];\R^2)$. By Claim \ref{Jlemm},  for 
 given $\om \ne 0$, the operator
 $\DD(\om,\la)\J(\om,\tau,\la)$ is  completely continuous from  $C_{2\pi}(\R \times [0,1];\R^2)$ into itself.
 Therefore, it remains to show that for $\om \ne 0$ and $\la \approx 0$
 the operator
 $I-\CC(\om,\la)-\DD(\om,\la)\KK(\la)$ is  Fredholm  of index zero from
 $C_{2\pi}(\R \times [0,1];\R^2)$ into itself. By Claim \ref{C1}, this is true
 whenever  
 the operator 
 $I-(I-\CC(\om,\la))^{-1}\DD(\om,\la)\KK(\la)$
 is  Fredholm  of index zero from
 $C_{2\pi}(\R \times [0,1];\R^2)$ into itself, for $\om \ne 0$ and $\la \approx 0$.
 For that we use the following Fredholmness criterion of S. M. Nikolskii (cf. e.g.  \cite[Theorem XIII.5.2]{Kant}):
 \begin{thm}
   \label{Niko}
   Let $U$ be a Banach space  and $K \in {\cal L}(U)$ be an operator  such that $K^2$ is completely continuous. Then
   the operator $I-K$ is Fredholm of index zero.
 \end{thm}

 On the account of Theorem \ref{Niko}, it remains to prove the 
 following statement.
\begin{claim}
For given $\om \ne 0$ and $\la \approx 0$,
the operator $ ((I-\CC(\om,\la))^{-1}\DD(\om,\la)\KK(\la))^2$
is completely continuous from $C_{2\pi}(\R \times [0,1];\R^2)$ into itself.	
\end{claim}
\begin{subproof}
A straightforward calculation shows that
\beq
\label{Cident}
 ((I-\CC)^{-1}\DD\KK))^2=
 (I-\CC)^{-1}\left((\DD\KK)^2+\DD\KK\CC(I-\CC)^{-1}\DD\KK\right).
 \ee
The desired statement now follows from 
Claims  \ref{DKDlemm} and \ref{DKClemm}.
\end{subproof} 
\begin{rem}
\label{estis}
For proving Lemma \ref{Fredia} we did not need the estimates \reff{sup1}, \reff{sup2}--\reff{I-C-1} and \reff{Cdelta}.
These estimates will be used in the proof of Lemma \ref{4.10} below (more exactly, in the proof of Claim 3 there).
\end{rem}

\subsection{Kernel and image of the linearization}
\label{Kernel and image of the linearization}
This subsection concerns the kernel and the image of the operator
\beq
\label{LLdef}
{\cal L}:=I-\CC-\DD(\J+\KK),
\ee
where
\beq
\label{CDdef}
\CC:=\CC(1,0),\;\DD:=\DD(1,0),\;\J:=\J(1,\tau_0,0), \mbox{ and } \KK:= \KK(0)
\ee
(cf. \reff{J12def}, \reff{K12def}, \reff{Cdef} and  \reff{Ddef}).
From now on we will use assumptions $\bf(A1)$--$\bf(A3)$ and \reff{Fred}
of Theorem \ref{thm:hopf}.
In particular, we will fix a solution $u=u_0\ne 0$ to  \reff{evp}  with $\tau=\tau_0$ and  $\mu=i$ and
a solution $u=u_*\ne 0$ to  \reff{ad} fulfilling assumption $\bf(A3)$
(or, more precisely, \reff{transvid} below).

We will describe the kernel and the image of the operator ${\cal L}$ by means of the eigenfunctions $u_0$ and $u_*$.
To this end, we introduce two  functions $v_0,v_*:[0,1] \to \C^2$, 
two functions $\vv_0,\vv _*:\R\times[0,1] \to \C^2$
and four functions  $v_0^1,v_0^2,v_*^1v_*^2:\R\times[0,1] \to \R^2$ by
\beq
\label{vnulldef}
v_0(x):=
\left[
  \begin{array}{c}
    iu_0(x)+a_0(x)u'_0(x)\\
    iu_0(x)-a_0(x)u'_0(x)
  \end{array}
\right], \;
\vv_0(t,x):=e^{it}v_0(x),\; v_0^1:=\mbox{Re}\,\vv_0,\; v_0^2:=\mbox{Im}\,\vv_0
\ee
and
\beq
\label{v*def}
v_*(x):=
\left[
  \begin{array}{c}
    u_*(x)+iU_*(x)\\
     u_*(x)-iU_*(x)
  \end{array}
\right], \;
\vv_*(t,x)
:=e^{it}v_*(x),\; v_*^1
:=\mbox{Re}\,\vv_*,\; v_*^2:=\mbox{Im}\,\vv_*,
\ee
where
\beq
\label{Udef}
U_*(x):=\left(\frac{b_6^0(x)}{a_0(x)}-2a_0'(x)\right)u_*(x)-a_0(x)u'_*(x)
+\frac{1}{a_0(x)}\int_x^1\left(b_3^0(\xi)+b_4^0(\xi)e^{i\tau_0}\right)u_*(\xi)d\xi.
\ee
\begin{lemma}\label{ker}
  If the conditions of Theorem \ref{thm:hopf}  
  are fulfilled,
then $\ker {\cal L}=\mbox{\rm span}\{v^1_0,v^2_0\}$. 
 \end{lemma}
 \begin{proof}
 	Because $u_0$ is a solution to  \reff{evp} with $\tau=\tau_0$ and $\mu=i$,
   the complex-valued function
   $
   \uu_0(t,x):=e^{it}u_0(x)
   $
  is a solution to the linear homogeneous problem
\beq
\label{linproblem}
\left.
\begin{array}{l}
  \d_t^2u(t,x)- a_0(x)^2\d_x^2u(t,x)\\
  =b_3^0(x)u(t,x)+b_4^0(x)u(t-\tau_0,x)+b_5^0(x)\partial_tu(t,x)+b_6^0(x)\partial_xu(t,x),\\
u(0,t) = \d_xu(t,1)=0,\;u(t+2\pi,x)=u(t,x).
\end{array}
\right\}
\ee
On the other hand, if $u$ is a solution to \reff{linproblem}, then for all $k \in \Z$ the functions
  $$
  \tilde{u}_k(x):=\frac{1}{2\pi}\int_0^{2\pi}u(t,x)e^{-ikt}dt
  $$
  satisfy the ODE 
$\left(-k^2-b_3^0(x)-b_4^0(x)e^{ik\tau_0}-ikb_5^0(x)\right)\tilde{u}_k(x)=a_0(x)^2\tilde{u}_k''(x)+b_6^0(x)\tilde{u}_k'(x)$
with boundary conditions $\tilde{u}_k(0)=\tilde{u}_k'(1)=0$.
    Assumptions $\bf(A1)$ and  $\bf(A2)$ imply that
     $\tilde{u}_k=0$ for all $k \in \Z\setminus\{\pm 1\}$ and $\tilde{u}_1=c
    u_0$ for some constant $c$,
    i.e., $u\in\mbox{span}\{\uu_0,\overline{\uu_0}\}$. In other words,
    $\mbox{span}\{\uu_0,\overline{\uu_0}\}$ consists of
     all solutions $u:\R \times [0,1] \to \C$ to \reff{linproblem}.

    Now we apply Lemmas \ref{lem:FOS} and \ref{lem:parint} with $\om=1$, $\tau=\tau_0$, $\la=0$ and with $b(x,\la,u_3,u_4,u_5,u_6)$ replaced by
    $b_3^0(x)u_3+b_4^0(x)u_4+b_5^0(x)u_5+b_6^0(x)u_6$. We conclude that 
     $\mbox{span}\{\vv_0,\overline{\vv_0}\}$ consists
    of all solutions $v:\R \times [0,1] \to \C^2$ to
    the linear homogeneous equation
    $
    v=\CC v-\DD(\J+\KK)v,
    $
    where $\vv_0$ is defined by \reff{vnulldef}.
    As $v_0^1=\mbox{Re}\vv_0$ and  $v_0^2=\mbox{Im}\vv_0$,  the proof is 
    complete.
  \end{proof}

  In what follows we denote by ``$\cdot$'' the Hermitian scalar product in $\C^2$, i.e.
  $v\cdot w:=v_1\overline{w_1}+v_2\overline{w_2}$ for  $v,w \in \C^2$.
  Further, for continuous functions $v,w:[0,2\pi]\times[0,1]\to\C^2$ we write
$$
\langle v,w \rangle:=\frac{1}{2\pi}\int_0^{2\pi}\int_0^1v(t,x)\cdot w(t,x)dxdt
=\frac{1}{2\pi}\int_0^{2\pi}\int_0^1(v_1(t,x)\overline{w_1(t,x)}+v_2(t,x)\overline{w_2(t,x)})dxdt.
$$
Moreover, we will work with the operator
$\A \in {\cal L}\left(C^1_{2\pi}(\R \times [0,1];\R^2);C_{2\pi}(\R \times [0,1];\R^2)\right)$, the components of which are defined by
$$
\left.
\begin{array}{rcl}
  \displaystyle[A_1v](t,x)&:=&\d_tv_1(t,x)-a_0(x)\d_xv_1(t,x)-b^0_1(x)v_1(t,x),\\
  \displaystyle[\A_2v](t,x)&:=&\d_tv_2(t,x)+a_0(x)\d_xv_2(t,x))-b^0_2(x)v_2(t,x)
\end{array}
\right\}
\; b_j^0(x):=b_j(x,0), j=1,2,
$$
i.e. $\A=\A(1,0)$ (cf. \reff{AAdef}),
and its formal adjoint one $\A^* \in {\cal L}\left(C^1_{2\pi}(\R \times [0,1];\R^2);C_{2\pi}(\R \times [0,1];\R^2)\right)$, which is defined by
\begin{eqnarray*}
  [\A_1^*v](t,x)&:=&-\d_tv_1(t,x)+\d_x(a_0(x)v_1(t,x))-b^0_1(x)v_1(t,x),\\
   \displaystyle[\A^*_2v](t,x)&:=&-\d_tv_2(t,x)-\d_x(a_0(x)v_2(t,x))-b^0_2(x)v_2(t,x).
\end{eqnarray*}
It is easy to verify that $\langle \A v,w\rangle=\langle v,\A^*w\rangle$ for all $v,w \in C^1_{2\pi}(\R \times [0,1];\R^2)$ which
satisfy the boundary conditions in \reff{FOS}.

\begin{lemma}\label{kerim}
	 If the conditions of Theorem \ref{thm:hopf}  are  fulfilled,
  then
  $$
  \mbox{\rm im}\,{\cal L}=\{f \in C_{2\pi}(\R \times [0,1];\R^2): \; \langle f,\A^*v_*^1\rangle=\langle f,\A^*v_*^2\rangle=0\}.
  $$
 \end{lemma}
\begin{proof}
  It follows from Lemmas \ref{Fredia} and \ref{ker} that  $\mbox{\rm im}\,{\cal L}$ is a closed subspace  of 
codimension two in  $C_{2\pi}(\R\times[0,1];\R^2)$. Hence,  it suffices to show that
  \beq
  \label{subset}
  \mbox{\rm im}\,{\cal L}\subseteq\{f \in C_{2\pi}(\R \times [0,1];\R^2): \; \langle f,\A^*v_*^1\rangle=\langle f,\A^*v_*^2\rangle=0\}
  \ee
  and that
   \beq
   \label{codim}
   \mbox{$\A^*v_*^1$ and $\A^*v_*^2$ are linearly independent.}
   \ee

   To prove \reff{subset}, fix an arbitrary $v \in C_{2\pi}(\R\times[0,1];\R^2)$. There exists
   a sequence $w^1,w^2,\ldots \in C^1_{2\pi}(\R\times[0,1];\R^2)$ such that $\|v-w^k\|_\infty \to 0$ as $k\to \infty$.
   Moreover, the functions $(\CC+\DD(\J+\KK))w^k$ satisfy the boundary conditions in \reff{FOS}. Also  the function $v_*^1$  satisfies the boundary 
conditions in \reff{FOS}. The last fact follows from the equalities
   $[(u_* +iU_*)+(u_* -iU_*)]_{x=0}=2u_*(0)=0$ and
   \beq
   \label{adbc}
   [(u_* +iU_*)-(u_* -iU_*)]_{x=1}=2i\left[\left(\frac{b_6^0}{a_0}-2a_0'\right)u_*-a_0u_*'\right]_{x=1}=0
   \ee
   because the eigenfunction $u_*$  satisfies
    the boundary conditions  in \reff{ad}. 
   Therefore, by  \reff{Aid1},
   $$
   \langle(\CC+\DD(\J+\KK))w^k,\A^*v_*^1\rangle= \langle\A(\CC+\DD(\J+\KK))w^k,v_*^1\rangle=
   \langle(\J+\KK)w^k,v_*^1\rangle= \langle w^k,(\J^*+\KK^*)v_*^1\rangle,
   $$
   where the operators
    $\J^*,\KK^*\in {\cal L}(C_{2\pi}(\R\times[0,1];\R^2))$
   are the formal adjoint operators to  $\J$ and $\KK$.
   Due to  \reff{J12def} and \reff{K12def}, they are given by the formulas
   $$
   [\J^*_1w](t,x)= - [\J^*_2w](t,x)=\frac{1}{2a_0(x)}\int_x^1(b_3^0(\xi)(w_1(t,\xi)+w_2(t,\xi))+b_4^0(\xi)(w_1(t+\tau_0,\xi)+w_2(t+\tau_0,\xi)))d\xi
   $$
   and
   $$
   \KK^*_1w=b_1^0w_2,\;  \KK^*_2w=b_2^0w_1,
   $$
   respectively. It follows that
   \begin{eqnarray*}
     &&\langle {\cal L}v,\A^*v_*^1\rangle=\langle (I-\CC-\DD(\J+\KK))v,\A^* v_*^1\rangle
     =\langle v,\A^* v_*^1\rangle-\lim_{k \to \infty}\langle (\CC+\DD(\J+\KK))w^k,\A^* v_*^1\rangle\\
     &&=\langle v,\A^* v_*^1\rangle-\lim_{k \to \infty}\langle ((\J+\KK)w^k,v_*^1\rangle=\langle v,(\A^*-\J^*-\KK^*)v_*^1\rangle.
   \end{eqnarray*}
   Similarly,  $\langle {\cal L}v,\A^*v_*^2\rangle=\langle v,(\A^*-\J^*-\KK^*)v_*^2\rangle$.
   Hence, in order to prove \reff{subset} it suffices to show that
   \beq
   \label{veq} 
   (\A^*-\J^*-\KK^*)\vv_*=0. 
   \ee

   Taking into account  the definitions of the operators $\A^*$, $\J^*$ and $\KK^*$ and of the function $v_*$ (cf. \reff{v*def}),
   it is easy to see that \reff{veq} is satisfied if and only if, for any $x \in [0,1]$,
      \begin{eqnarray*}
       &&\left[-iv_{*1}+(a_0v_{*1})'-b_1^0(v_{*1}+v_{*2})\right](x)=\frac{1}{2a_0(x)}\int_x^1(b_3^0(\xi)+b_4^0(\xi)
         e^{i\tau_0})(v_{*1}(\xi)+v_{*2}(\xi))d\xi,\\
       &&\left[-iv_{*2}-(a_0v_{*2})'-b_2^0(v_{*1}+v_{*2})\right](x)=-\frac{1}{2a_0(x)}\int_x^1(b_3^0(\xi)+b_4^0(\xi)
         e^{i\tau_0})(v_{*1}(\xi)+v_{*2}(\xi))d\xi,
   \end{eqnarray*}
   where $v_{*1}=u_*+iU_*$ and  $v_{*2}=u_*-iU_*$ are the components of the vector function $v_*$.
   Considering the sum and the difference of these two 
  equations 
   and taking into account that $v_{*1}+v_{*2}=2u_*$ and  $v_{*1}-v_{*2}=2iU_*$,
   we get
   \begin{eqnarray}
     \label{s1}
     &&-iu_{*}(x)+i(a_0U_{*})'(x)-(b_1^0(x)+b_2^0(x))u_*(x)=0,\\
     \label{s2}
      &&U_*(x)+(a_0u_{*})'(x)-(b_1^0(x)-b_2^0(x))u_*(x)=\frac{1}{a_0(x)}\int_x^1(b_3^0(\xi)+b_4^0(\xi)e^{i\tau_0})u_*(\xi)d\xi.
   \end{eqnarray}

Thus, \reff{veq} is equivalent  to \reff{s1}--\reff{s2}.
   In order to show  \reff{s1},  we use the equality $b_1^0+b_2^0=b_5^0$ (cf. \reff{b1def}) and note  that  \reff{s1} is equivalent to
   \beq
   \label{s3}
   (a_0U_*)'=(1-ib_5^0)u_*.
   \ee
   On the other side,  \reff{Udef} yields
\beq
   \label{s4}
   (a_0U_*)'=(b_6^0u_*)'-(a_0^2u_*)''-(b_3^0+b_4^0e^{i\tau_0})u_*.
   \ee
   Inserting \reff{s4} into \reff{s3}, we conclude that \reff{s1} is true if
   $u_*$ solves the ordinary differential equation in  \reff{ad}, i.e. \reff{s1} is satisfied.

Equation \reff{s2} is satisfied by the definition \reff{Udef} of the function  $U_*$ and  the equality
   $b_1^0-b_2^0=-a_0'+b_6^0/a_0$ (cf. \reff{b1def}). The proof of
   \reff{veq} and, hence, of
    \reff{subset}
   is therefore complete.

   It remains to prove \reff{codim}. To this end, we introduce functions $w_0: [0,1] \to \C^2$,
   $\ww_0: \R \times[0,1] \to \C^2$ and $w^1,w^2 \in C^1_{2\pi}(\R \times [0,1];\R^2)$ by
   \beq
   \label{w0def}
w_0:=\left[
     \begin{array}{c}
       (i+\tau_0b_4^0e^{-i\tau_0})u_0+a_0u_0'\\
        (i+\tau_0b_4^0e^{-i\tau_0})u_0-a_0u_0'
     \end{array}
   \right],
   \quad \ww_0(t,x):=e^{it}w_0(x)
   \ee
   and
   \beq
   \label{wdef}
   (I-\CC)w^1=\DD\,\mbox{Re}\,\ww_0,\quad  (I-\CC)w^2=-\DD\,\mbox{Im}\,\ww_0.
   \ee
   Note that the equations \reff{wdef} define the functions $w^1,w^2 \in C^1_{2\pi}(\R \times [0,1];\R^2)$ uniquely,
   as follows from Claim \ref{C1} in Section 
   \ref{Fredholmness of the linearization}
   (see also Remark \ref{c1}).
  Combining  \reff{Aid1} with \reff{wdef}, we obtain 
 $$
 \A w^1=\mbox{Re}\,\ww_0,\; \A w^2=-\mbox{Im}\,\ww_0.
$$
Therefore,
   \begin{eqnarray}
     \label{int1}
     &&\langle w^1,\A^*\vv_* \rangle= \langle \A w^1,\vv_*^1 \rangle= \langle \mbox{Re}\,\ww_0,\vv_*\rangle\nonumber\\
     &&=\frac{1}{4\pi}\int_0^{2\pi}\int_0^1\left(e^{it}w_0(x)+e^{-it}\overline{w_0(x)}\right)\cdot
        e^{it}v_*(x)\,dxdt=\frac{1}{2}\int_0^1w_0(x)\cdot 
        v_*(x)\,dx.
   \end{eqnarray}
   By 
\reff{v*def} and 
\reff{w0def}, the right hand side of \reff{int1} is equal to
    \begin{eqnarray*}
     &&\frac{1}{2}\int_0^1\left(((i+\tau_0b_4^0e^{-i\tau_0})u_0+a_0u_0')\cdot(\overline{u_*}-i\overline{U_*})+
        ((i+\tau_0b_4^0e^{-i\tau_0})u_0-a_0u_0')\cdot(\overline{u_*}+i\overline{U_*})\right)dx\\
     &&=\int_0^1((i+\tau_0b_4^0e^{-i\tau_0})u_0\overline{u_*}-ia_0u_0'\overline{U_*})dx
        =\int_0^1((i+\tau_0b_4^0e^{-i\tau_0})u_0\overline{u_*}+iu_0(a_0\overline{U_*})'\,dx.
    \end{eqnarray*}
Finally, we use  \reff{Udef} and the definition of $\sigma$ in {\bf(A3)} to get
    $$
      \langle w^1,\A^*v_*\rangle
      =
      \int_0^1(2i+\tau_0b_4^0e^{-i\tau_0}-b_5^0)u_0\overline{u_*}\,dx=\sigma.
      $$
Similarly,
$$
     \langle w^2,\A^*\vv_* \rangle= -\langle \mbox{Im}\,\ww_0,\vv_*\rangle
     =-\frac{1}{4\pi i}\int_0^{2\pi}\int_0^1\left(e^{it}w_0-e^{-it}\overline{w_0}\right)\cdot
     e^{it}{v_*}\,dxdt=-\frac{1}{2i}\int_0^1w_0\cdot{v_*}\,dx=i\sigma.
     $$
      Now, we normalize  the eigenfunctions $u_0$ and $u_*$ so that
      \beq
   \label{transvid}
   \sigma=\int_0^1\left(2i-b^0_5(x)+\tau_0e^{-i\tau_0}b^0_4(x)\right)u_0(x)\overline{u_*(x)}dx=1.
   \ee
It follows that
\beq
       \label{delta}
       \langle w^j,\A^*v_*^k \rangle=\delta^{jk},
       \ee
       which yields \reff{codim}, as desired.
   \end{proof}

\subsection{Splitting of equation \reff{abstract}}
\label{splitting}
Given $\vp \in \R$, we introduce a time shift operator $S_\vp \in \LL(C_{2\pi}(\R\times[0,1];\R^2))$ by
\beq
\label{S}
[S_\vp v](t,x):=v(t+\vp,x).
\ee
It is easy to verify that 
\begin{eqnarray}
  \label{inv1}
  &&S_\vp \A(\om,\la)=\A(\om,\la)S_\vp,\;
  S_\vp \CC(\om,\la)=\CC(\om,\la)S_\vp,\;
     S_\vp \DD(\om,\la)=\DD(\om,\la)S_\vp
  \end{eqnarray}   
  and
 \begin{eqnarray}    
  \label{inv2}
  &&S_\vp \B(v,\om,\tau,\la)=\B(S_\vp v,\om,\tau,\la) 
\end{eqnarray}
for all $\vp, \om, \tau, \la \in \R$ and $v \in C_{2\pi}(\R\times[0,1];\R^2))$.
It follows that $S_\vp \LL=\LL S_\vp$, in particular,
\beq
\label{invspace}
S_\vp \ker \LL=\ker \LL, \; S_\vp \mbox{im}\,\LL=\mbox{im}\, \LL.
\ee
Since $\ker \LL$ is finite dimensional, there exists a topological complement $\W$ (i.e., a closed subspace which is transversal) to
$\ker \LL$ in $ C_{2\pi}(\R\times[0,1];\R^2))$. Since the map $\vp \in \R \mapsto S_\vp \in \LL(C_{2\pi}(\R\times[0,1];\R^2))$
is strongly continuous, $\W$ 
can be chosen to be invariant with respect to  $S_\vp$, i.e.,
\beq
\label{split}
C_{2\pi}(\R\times[0,1];\R^2)=\ker \LL \oplus \W \ \mbox{ and }  \  S_\vp\W=\W
\ee
(cf. \cite[Theorem 2]{Dancer}).
Further, let us introduce a projection operator $P \in  \LL(C_{2\pi}(\R\times[0,1];\R^2))$ by
\beq
\label{Pdef}
Pv:=\langle v,\A^*v_*^1\rangle w^1+\langle v,\A^*v_*^2\rangle w^2,
\ee
where the functions $v_*^j$ and $w^k$ are given by \reff{vnulldef} and \reff{wdef}.
The projection property $P^2=P$ follows from \reff{delta}.
Moreover, Lemma \ref{kerim} implies that
\beq
\label{kerP}
\ker P= \mbox{im}\,\LL.
\ee
Furthermore, from \reff{w0def} it follows that $[S_\vp \ww_0](t,x)=\ww_0(t+\vp,x)=e^{i(t+\vp)}\ww_0(x)=e^{i\vp}\ww_0(x)$
and, hence,
$$
S_\vp \mbox{Re}\,\ww_0=\cos \vp\, \mbox{Re}\,\ww_0-\sin \vp \,\mbox{Im}\,\ww_0\  \mbox{ and } \ 
S_\vp \mbox{Im}\,\ww_0=\cos \vp \,\mbox{Im}\,\ww_0+\sin \vp \,\mbox{Re}\,\ww_0.
$$
Similarly one shows that 
$
S_\vp v_*^1=\cos \vp \,v_*^1-\sin \vp \,v_*^2$ and $
S_\vp v_*^2=\cos \vp \,v_*^2+\sin \vp \,v_*^1.
$
On the account of \reff{inv1},
 for every $v \in C_{2\pi}(\R\times[0,1];\R^2)$ we obtain
\begin{eqnarray}
  \label{Pinv}
  &&PS_\vp v=\langle S_\vp v,\A^*v_*^1\rangle w^1+\langle S_\vp v,\A^*v_*^2\rangle w^2=
  \langle v,\A^*S_{-\vp}v_*^1\rangle w^1+\langle v,\A^*S_{-\vp}v_*^2\rangle w^2\nonumber \\
  &&=\left(\cos\vp\langle v,\A^*v_*^1\rangle+\sin\vp\langle v,\A^*v_*^2\rangle\right)w^1
     +\left(-\sin\vp\langle v,\A^*v_*^1\rangle+\cos\vp\langle v,\A^*v_*^2\rangle\right)w^2\nonumber\\
  &&=\langle v,\A^*v_*^1\rangle\left(\cos \vp\, w^1-\sin \vp \,w^2\right)+
     \langle v,\A^*v_*^2\rangle\left(\sin \vp\, w^1+\cos \vp \,w^2\right)\nonumber\\
  &&=\langle v,\A^*v_*^1\rangle S_\vp w^1+\langle v,\A^*v_*^2\rangle S_\vp w^2= S_\vp Pv.
     \end{eqnarray}
Finally, we use the ansatz (cf. \reff{split})
\beq
\label{ansatz}
v=u+w,\quad u \in \ker \LL,\; w \in \W
\ee
    and rewrite equation \reff{abstract} as a system of two equations, namely
     \begin{eqnarray}
       \label{in}
       P\left((I-\CC(\om,\la))(u+w)-\DD(\om,\la)\B(u+w,\om,\tau,\la)\right)=0,\\
       \label{ex}
        (I-P)\left((I-\CC(\om,\la))(u+w)-\DD(\om,\la)\B(u+w,\om,\tau,\la)\right)=0.
     \end{eqnarray}

     \section{The external Lyapunov-Schmidt equation}
     \label{sec:external}
In this section we solve
the so-called external Lyapunov-Schmidt equation \reff{ex} with respect to $w \approx 0$ for $u\approx 0$,
$\om\approx 1$, $\tau\approx \tau_0$ and $\la \approx 0$.
More exactly, in Subsection \ref{appendix}  we present a generalized implicit function theorem,
    which will be used in Subsection~\ref{Solution of the external Lyapunov-Schmidt equation}   to solve equation  \reff{ex}.

   \subsection{A generalized implicit function theorem}\label{appendix}
    \renewcommand{\theequation}{{\thesection}.\arabic{equation}}
  \setcounter{equation}{0}
  In this subsection we present the generalized implicit function theorem,
    which 
  is a particular case of \cite[Theorem 2.2]{KR4}.
  It concerns abstract parameter-dependent equations of the type
  \beq
  \label{abst}
  F(w,p)=0.
  \ee
  Here $F$ is a map from
   ${\cal W}_0 \times {\cal P}$ to $\widetilde{\cal W}_0$, ${\cal W}_0$ and  ${\widetilde{\cal W}}_0$ are Banach spaces with norms $\|\cdot\|_0$
  and $|\cdot|_0$, respectively, and
  ${\cal P}$ is a finite dimensional  normed vector space with norm $\|\cdot\|$. Moreover, it is supposed that
  \beq
  \label{null}
  F(0,0)=0.
  \ee
  We are going to
  state conditions on $F$ such 
  that, similarly to the classical implicit function theorem,
  for all $p \approx 0$ there exists exactly one solution $w \approx 0$ to \reff{abst}
  and that the data-to-solution map $p \mapsto w$ is smooth.
  Similarly to the classical implicit function theorem, we suppose that
  \beq
  \label{Fsmooth}
  F(\cdot,p) \in C^\infty({\cal W}_0;{\widetilde{\cal W}}_0) \mbox{ for all } p \in {\cal P}.
  \ee
  However, unlike  to the classical case, we do not suppose that  $F(w,\cdot)$ is smooth for all $w \in {\cal W}_0$.
  In our applications the map $(w,p) \mapsto\d_wF(w,p)$ is not even  continuous with respect to the uniform operator norm
  in ${\cal L}({\cal W}_0;{\widetilde{\cal W}}_0)$, in general.
  Hence, the difference of  Theorem \ref{thm:IFT} below
  in comparison with the classical  implicit function theorem is not a degeneracy of the partial derivatives  $\d_wF(w,p)$
  (like in  implicit function theorems of Nash-Moser type), but a  degeneracy of the partial derivatives  $\d_pF(w,p)$
  (which do not exist for all $w \in {\cal W}_0$). 
  
  Thus, we consider parameter depending equations, 
  which do not depend smoothly on the parameter, but with solutions which do  depend smoothly on the parameter.
  For that, of course, some additional structure is needed, which will be described now.
  
  Let $\vp \in \R \mapsto S(\vp) \in \LL({\cal W}_0)$,
  $\vp \in \R \mapsto \widetilde{S}(\vp) \in \LL(\widetilde{{\cal W}}_0)$,
  and  $\vp \in \R \mapsto T(\vp) \in \LL({\cal P})$
  be strongly continuous groups of linear bounded operators on ${\cal W}_0$, $\widetilde{{\cal W}}_0$ and  ${\cal P}$, respectively.
  We suppose that
  \beq
  \label{symm}
  \widetilde{S}(\vp)F(w,p)=F(S(\vp)w,T(\vp)p) \mbox{ for all } \vp \in \R, w \in {\cal W}_0 \mbox{ and }  p \in {\cal P}.
  \ee
  Furthermore, let $A:D(A) \subseteq  {\cal W}_0 \to  {\cal W}_0$ be the infinitesimal generator of the $C_0$-group $S(\vp)$.
  For $l\in \N$, let
  $$
  {\cal W}_l:=D(A^l)=\{w \in  {\cal W}_0: S(\cdot)w \in C^l(\R; {\cal W_0})\}
  $$
  denote the domain of definition of the $l$-th power of $A$. Since $A$ is closed, ${\cal W}_l$ is a Banach space with the norm
  $$
  \|w\|_l:=\sum_{k=0}^l\|A^kw\|_0.
  $$
  We suppose that for all $k,l\in \N$
  \beq
  \label{infsmooth}
  \d^k_w\F(w,\cdot)(w_1,\ldots,w_k) \in C^l({\cal P};{\widetilde{\cal W}}_0) \mbox{ for all } w,w_1,\ldots,w_k \in {\cal W}_l
  \ee 
  and, for all $w,w_1,\ldots,w_k \in  {\cal W}_l$ and $p,p_1,\ldots,p_l \in  {\cal P}$ with $\|w\|_l+\|p\| \le 1$,
  \beq
  \label{esti}
  |\d^l_p\d^k_w\F(w,p)(w_1,\ldots,w_k,p_1,\ldots,p_l)|_0 \le c_{kl}\|w_1\|_l\ldots\|w_k\|_l\,\|p_1\|\ldots\|p_l\|,
  \ee
  where the constants $c_{kl}$ do not depend on  $w,w_1,\ldots,w_k,p,p_1,\ldots,p_l$.
  \begin{thm}\cite[Theorem 2.2]{KR4}
  	\label{thm:IFT}
  	Suppose that the conditions \reff{null}--\reff{esti} are fulfilled. Furthermore, assume that there exist $\eps_0 >0$ and $c>0$ such that 
  	for all  $p \in {\cal P}$ with $\|p\| \le \eps_0$
  	\beq
  	\label{FH}
  	\d_wF(0,p) \mbox{ is Fredholm of index zero from ${\cal W}_0$ into $\widetilde{\cal W}_0$}
  	\ee
  	and 
  	\beq
  	\label{coerz}
  	|\d_wF(0,p)w|_0 \ge c\|w\|_0 \mbox{ for all } w \in {\cal W}_0.
  	\ee
  	Then there exist $\eps \in (0,\eps_0]$ and $\delta>0$ such that for all  $p \in {\cal P}$ with $\|p\| \le \eps$ there is
  	a unique solution $w=\hat{w}(p)$ to \reff{abst} with $\|w\|_0 \le \delta$. Moreover, for all $k\in \N$
  	we have $\hat{w}(p) \in {\cal W}_k$, and the map $p \in  {\cal P}\mapsto \hat{w}(p) \in {\cal W}_k$ is $C^\infty$-smooth.
  \end{thm}
  \begin{rem}
  	\label{hardIFT}
  	The maps $\vp \in \R \mapsto S(\vp) \in {\cal L}({\cal W}_0)$ and
  	$\vp \in \R \mapsto \widetilde{S}(\vp) \in {\cal L}(\widetilde{\cal W}_0)$
  	are not continuous, in general.
  	Nevertheless, since ${\cal P}$ is supposed to be finite dimensional, the map $\vp \in \R \mapsto T(\vp) \in {\cal L}({\cal P})$ is $C^\infty$-smooth.
  	This is essential in the proof of Theorem \ref{thm:IFT} in \cite{KR4}.
      \end{rem}
       \begin{rem}
         \label{Lanull}
         In Theorem \ref{thm:IFT} we do not suppose that $\d_wF(0,p)$ depends continuously on $p$ in the sense of the uniform operator norm
         in ${\cal L}({\cal W}_0;{\widetilde{\cal W}}_0)$. Hence, assumptions \reff{FH} and \reff{coerz} cannot be replaced by their versions with $p=0$, in general.
         \end{rem}

  \subsection{ Solution of the external Lyapunov-Schmidt equation}
  \label{external}
  \setcounter{claim}{0}
\label{Solution of the external Lyapunov-Schmidt equation}
In what follows, 
we  use the following notation (for $\eps>0$ and $k\in \N$):
\begin{eqnarray*}
  &&\UU_\eps:=\{u \in \ker \LL:\; \|u\|_\infty <\eps\},\quad \PP_\eps:=\{(\om,\tau,\la) \in \R^3:\; |\om-1|+|\tau-\tau_0|+|\la|<\eps\},\\
  &&C_{2\pi}:=C_{2\pi}(\R\times[0,1];\R^2),\quad C^k_{2\pi}:=C^k_{2\pi}(\R\times[0,1];\R^2).
\end{eqnarray*}

We are going to solve
the so-called external Lyapunov-Schmidt equation \reff{ex} with respect to $w \approx 0$ for $u\approx 0$,
$\om\approx 1$, $\tau\approx \tau_0$ and $\la \approx 0$.
\begin{lemma}
  \label{4.10}
   Let  the conditions of Theorem \ref{thm:hopf}  
  be fulfilled.
 Then there exist $\eps>0$ and $\delta>0$
  such that for all $u \in \UU_\eps$ and $(\om,\tau,\la)\in \PP_\eps$ there is a unique  solution $w=\hat{w}(u,\om,\tau,\la) \in \W$
  to \reff{ex} with $\|w\|_\infty <\delta$.
  Moreover, for all $k\in \N$ it holds $\hat{w}(u,\om,\tau,\la) \in C^k_{2\pi}$, and the map
  $(u,\om,\tau,\la) \in  \UU_\eps\times\PP_\eps\mapsto \hat{w}(u,\om,\tau,\la) \in C^k_{2\pi}$ is $C^\infty$-smooth.
\end{lemma}

We have that $w=0$ is a solution to \reff{ex} with $u=0$, $\om=1$, $\tau=\tau_0$ and $\la=0$. This suggests that Lemma \ref{4.10} 
can be obtained from an appropriate
implicit function theorem. Unfortunately,
the classical implicit function theorem does not work here, because
the left-hand side of \reff{ex} is differentiable with respect to $\om$, $\tau$ and $\la$ not for any $w\in C_{2\pi}$.
We will apply Theorem \ref{thm:IFT}.

 Let us verify the assumptions of  Theorem  \ref{thm:IFT} in the
 following setting:
\beq
\label{setting}
\left.
  \begin{array}{l}
    \W_0=\W,\;  \widetilde{\W}_0=\mbox{im}\,P,\; \PP=\ker \LL \times \R^3, \; p=(u,\om-1,\tau-\tau_0,\la),\\
    F(w,p)=(I-P)\left((I-\CC(\om,\la))(u+w)-\DD(\om,\la)\B(u+w,\om,\tau,\la)\right).
  \end{array}
  \right\}
  \ee
Note that  $\W_0$ and   $\widetilde{\W}_0$ are Banach spaces with the norm $\|\cdot\|_\infty$. Conditions    \reff{null}, \reff{Fsmooth} and 
\reff{FH} are fulfilled, the last one being true due to  Lemma \ref{Fredia}.

  It remains to verify  conditions \reff{symm}--\reff{esti} and \reff{coerz}.

  We begin with verifying \reff{symm}. 
We identify $S_\vp$ and  $\widetilde{S}_\vp$ with $S_\vp$ defined by \reff{S} 
  restricted to $\W_0$ and  $\widetilde{\W}_0$,
  respectively.  Let
  $$
  T_\vp(u,\om,\tau,\la):=(S_\vp u,\om,\tau,\la).
  $$
  It follows from \reff{Pinv} that $S_\vp\widetilde{\W}_0=\widetilde{\W}_0$. 
  Taking into account \reff{inv1} and \reff{inv2}, we get
  \begin{eqnarray}
    \label{Finv}
    \widetilde{S}_\vp F(w,p)&=&S_\vp (I-P)\left((I-\CC(\om,\la))(u+w)-\DD(\om,\la)\B(u+w,\om,\tau,\la)\right)\nonumber\\
    &=&(I-P)\left((I-\CC(\om,\la))(S_\vp u+S_\vp w)-\DD(\om,\la)\B(S_\vp u+S_\vp w,\om,\tau,\la)\right)\nonumber\\
    &=&F(S_\vp w,T_\vp p),
    \end{eqnarray}
 which gives \reff{symm}.

 To verify assumption \reff{infsmooth},
 recall that the infinitesimal generator of the group $S_\vp$ is the
 differential operator $A=\frac{d}{dt}$. Therefore,
$$
\W^l=\{w \in \W: \; \partial_tw,\partial_t^2w,\ldots,\partial_t^lw\in \W\},\;
\|w\|_l=\sum_{j=0}^l\|\partial_t^jw\|_\infty \mbox{ for } w \in \W^l.
$$
We have
\begin{eqnarray*}
  &&\partial_w[(I-P)\left((I-\CC(\om,\la))(u+w)-\DD(\om,\la)\B(u+w,\om,\tau,\la)\right)]w_1\\
  &&=(I-P)\left(I-\CC(\om,\la)-\DD(\om,\la)\partial_v\B(u+w,\om,\tau,\la)\right)w_1
\end{eqnarray*}
and
\begin{eqnarray*}
  &&\partial^k_w(I-P)\left((I-\CC(\om,\la))(u+w)-\DD(\om,\la)\B(u+w,\om,\tau,\la)\right)(w_1,\ldots,w_k)\\
  &&=-(I-P)\DD(\om,\la)\partial^k_v\B(u+w,\om,\tau,\la)(w_1,\ldots,w_k) \mbox{ for } k \ge 2.
\end{eqnarray*}
Taking into account that  any $u \in \ker \LL$ is $C^\infty$-smooth and 
satisfies the equality
$\|\partial_t^ju\|_\infty=\|u\|_\infty$   (cf. Lemma \ref{ker}), our task is reduced 
to show that for all $k,l\in\N$ and all $w,w_1,\ldots,w_k \in \W^l$ the functions
$\CC(\om,\la)w$ and 
$\DD(\om,\la)\partial^k_v\B(u+w,\om,\tau,\la)(w_1,\ldots,w_k)$
depend $C^l$-smoothly 
 on $(\om,\tau,\la)$
  and that condition  \reff{esti} is fulfilled.

The proof goes through two claims.
\begin{claim}
  \label{4.11}
  For all $l,m\in\N$ and $w \in \W^{l+m}$ the map $(\om,\la) \in \R^2 \mapsto \CC(\om,\la) w \in C_{2\pi}$ is $C^{l+m}$-smooth. Moreover, 
  \beq
  \label{apri1}
    \|\partial^l_\om\partial^m_\la\CC(\om,\la) w\|_\infty \le c_{lm}\|w\|_{l+m},
    \ee
    where the constant $c_{lm}$ does not depend on $\om$, $\la$ and $w$
     for $\om$ and $\la$ varying on bounded intervals.
\end{claim}
  \begin{subproof}
  	Since $w(\cdot,x)$ is  $C^l$-smooth,
     definition \reff{Cdef} implies that  $\CC(\cdot,\cdot)w$ is $C^l$-smooth,
    and the derivatives can be calculated by the chain rule. For example,
    \beq
    \label{loss}
    \partial_\om[\CC(\om,\la)w](t,x)=
    \left[
      \begin{array}{c}
        -c_1(x,0,\la)\partial_tw_2(t+\om A(x,0,\la),0)A(x,0,\la)\\
        -c_2(x,1,\la)\partial_tw_1(t-\om A(x,1,\la),1)A(x,1,\la)
      \end{array}
    \right].
    \ee
    It follows that $\|\partial_\om[\CC(\om,\la)w\|_\infty \le \mbox{const}\|w\|_1$, where the constant does not depend on
    $\om$ and $\la$ (varying in  bounded intervals) and on $w \in \W^1$.

    Similarly one can handle $\partial_\la\CC(\om,\la)w$ and higher order derivatives, and similarly one can show 
\reff{apri1}.
  \end{subproof}
  \begin{rem}
    \label{rem:loss}
    In \reff{loss} the loss of derivatives property can be seen explicitely: Taking a derivative with
    respect to $\om$ leads to a derivative with respect to $t$. The same happens in formulas \reff{Dform}, \reff{Jform} and \reff{bform} below.
    \end{rem}
\begin{claim}
  \label{4.12}
  For all $k,l,m,n\in\N$, $u \in \ker \LL$ and $w,w_1,\ldots,w_k \in \W^{l+m+n}$,
the map $(\om,\tau,\la) \in \R^3 \mapsto \DD(\om,\la)\partial_v^k\B(u+w,\om,\tau,\la) \in C_{2\pi}$ is $C^{l+m+n}$-smooth. 
Moreover,
\beq
\label{apri2}
    \|\partial^l_\om\partial^m_\tau\partial^n_\la[\DD(\om,\la)\partial_v^k\B(u+w,\om,\tau,\la)(w_1,\ldots,w_k)\|_\infty
    \le c_{klmn}\|w_1\|_{l+m+n}\cdot\ldots\cdot\|w_k\|_{l+m+n},
    \ee
    where the constant $c_{klmn}$ does not depend on $\om$, $\tau$, $\la$, $u$  and $w$ for $\|u\|_\infty$, $\|w\|_{l+m+n}$,
    $\om$, $\tau$ and
    $\la$  varying on bounded intervals.
  \end{claim}
  \begin{subproof}
    Differentiation of \reff{Ddef} with respect to $\om$ gives
    \beq
    \label{Dform}
    \partial_\om\DD(\om,\la)w=\widetilde{\DD}(\om,\la)\partial_tw \mbox{ for } w \in \W^1,
    \ee
    where
    $$
      \begin{array}{cc}
      [\widetilde{\DD}_1(\om,\la)w](t,x):=   \displaystyle-\int_0^x\frac{c_1(x,\xi,\la)}{a(\xi,\la)}w_1(t+\om A(x,\xi,\la),\xi)A(x,\xi,\la)\,d\xi,\\ [4mm]
      [\widetilde{\DD}_2(\om,\la)w](t,x):=    \displaystyle-\int_x^1\frac{c_2(x,\xi,\la)}{a(\xi,\la)}w_2(t-\om A(x,\xi,\la),\xi)A(x,\xi,\la)\,d\xi.
      \end{array}
    $$
    Hence, for  $v,w \in \W^1$, it holds
    \beq
    \label{Deq}
    \partial_\om[\DD(\om,\la)\partial_v\B(v,\om,\tau,\la)w]= \widetilde{\DD}(\om,\la)\partial_t[\partial_v\B(v,\om,\tau,\la)w]+
    \DD(\om,\la)\partial_\om[\partial_v\B(v,\om,\tau,\la)w].
    \ee
    Furthermore, similarly to \reff{diffB}, we have
    $$
    \partial_v\B(v,\om,\tau,\la)=\tilde{\J}(v,\om,\tau,\la)+\widetilde{\KK}(v,\om,\tau,\la),
    $$
    where
    $$
[\tilde{\J}_j(v,\om,\tau,\la)w](t,x):=\tilde{b}_3(t,x,v,\om,\tau,\la)[J_\la w](t,x)+
\tilde{b}_4(t,x,v,\om,\tau,\la)[J_\la w](t-\om\tau,x), \; j=1,2,
$$
and
$$
\begin{array}{cc}
[\widetilde{\KK}_1(v,\om,\tau,\la)w](t,x):=\tilde{b}_2(t,x,v,\om,\tau,\la)w_2(t,x),\\ [3mm]
[\widetilde{\KK}_2(v,\om,\tau,\la)w](t,x):=\tilde{b}_1(t,x,v,\om,\tau,\la)w_1(t,x).
\end{array}
$$
Here the coefficients $\tilde{b}_k$ are defined appropriately
(similarly to \reff{bdef1} and \reff{b1def}), as follows:
$$
\tilde{b}_k(t,x,v,\om,\tau,\la):=\partial_jb(x,\la,[J_\la v](t,x),[J_\la v](t-\om\tau,x),[Kv](t,x),[K_\la v](t,x)) \mbox{ for } k=3,4,5,6
$$
and
$$
  \begin{array}{rcl}
  \tilde  b_1(t,x,v,\om,\tau,\la)&:=&\displaystyle\frac{1}{2}\left(-\partial_xa(x,\la)
                                +\tilde{b}_5(t,x,v,\om,\tau,\la)+\frac{\tilde{b}_6(t,x,v,\om,\tau,\la)}{a(x,\la)}\right),\;\\
  \tilde  b_2(t,x,v,\om,\tau,\la)&:=&\displaystyle\frac{1}{2}\left(\partial_xa(x,\la)
                                +\tilde{b}_5(t,x,v,\om,\tau,\la)-\frac{\tilde{b}_6(t,x,v,\om,\tau,\la)}{a(x,\la)}\right).
  \end{array}
  $$
Now, $[\partial_v\B_j(v,\cdot,\cdot,\cdot)w](\cdot,x)$ is $C^l$-smooth
because $v(\cdot,x)$ and  $w(\cdot,x)$ are  $C^l$-smooth. The derivatives can be calculated by the product and chain rules.
In particular, for $v,w \in \W^1$ we have
 $$
    \partial_\om\partial_v\B(v,\om,\tau,\la)=\partial_\om\tilde{\J}(v,\om,\tau,\la)+\partial_\om\widetilde{\KK}(v,\om,\tau,\la),
    $$
    where
    \begin{eqnarray}
      \label{Jform}
      [\partial_\om\tilde{\J}_j(v,\om,\tau,\la)w](t,x)&=&
      \partial_\om\tilde{b}_3(t,x,v,\om,\tau,\la)[J_\la w](t,x)+
                                                          \partial_\om\tilde{b}_4(t,x,v,\om,\tau,\la)[J_\la w](t-\om\tau,x)\nonumber\\
      &&-\tau\tilde{b}_4(t,x,v,\om,\tau,\la)[J_\la \partial_tw](t-\om\tau,x)
    \end{eqnarray}
    and
\begin{eqnarray*}
&&[\partial_\om\widetilde{\KK}_1(v,\om,\tau,\la)w](t,x)=\frac{1}{2}\left(
                                \partial_\om \tilde{b}_5(t,x,v,\om,\tau,\la)+\frac{\partial_\om \tilde{b}_6(t,x,v,\om,\tau,\la)}{a(x,\la)}\right)w_2(t,x),\\
&&[\partial_\om\widetilde{\KK}_2(v,\om,\tau,\la)w](t,x)=\frac{1}{2}\left(
                                \partial_\om \tilde{b}_5(t,x,v,\om,\tau,\la)-\frac{\partial_\om \tilde{b}_6(t,x,v,\om,\tau,\la)}{a(x,\la)}\right)w_1(t,x).
\end{eqnarray*}
Moreover, for  $k=3,4,5,6$,
\begin{eqnarray}
\label{bform}
  &&\partial_\om\tilde{b}_k(t,x,v,\om,\tau,\la)\nonumber\\
  &&=
-\tau \partial_4\partial_jb(x,\la,[J_\la v](t,x),[J_\la v](t-\om\tau,x),[Kv](t,x),[K_\la v](t,x))[J_\la \partial_t v](t-\om\tau,x). 
\end{eqnarray}
The functions $\partial_\om\tilde{b}_k$ are bounded as long as as ´$\|v\|_1$ $\om$, $\tau$ and
 $\la$ are bounded.
 Hence, we have
 $$
 \|\partial_\om[\partial_v\B_j(v,\om,\tau,\la)w]\|_\infty \le \mbox{const}\|w\|_1,
 $$
  where the constant does not depend on $\om$, $\tau$, $\la$, $v$  and $w$ as long as $\|v\|_1$, $\|w\|_1$  $\om$, $\tau$ and
  $\la$ are bounded. Similarly one shows
   $
 \|\partial_t[\partial_v\B_j(v,\om,\tau,\la)w]\|_\infty \le \mbox{const}\|w\|_1.
 $
 Using \reff{Deq} we get
 $$
 \|\partial_\om[\DD(\om,\la)\partial_v\B(v,\om,\tau,\la)w]\|_\infty \le \mbox{const}\|w\|_1,
 $$
  where the constant does not depend on $\om$, $\tau$, $\la$, $v$  and $w$ as long as $\|v\|_1$, $\|w\|_1$  $\om$, $\tau$ and
  $\la$ are bounded. Similarly one shows the estimates \reff{apri2} for $\partial_\tau[\DD(\om,\la)\partial_v\B(v,\om,\tau,\la)w]$ and 
 $\partial_\la[\DD(\om,\la)\partial_v\B(v,\om,\tau,\la)w]$
and for the higer order derivatives.
\end{subproof}

Finally, we verify  assumption \reff{coerz} of  Theorem  \ref{thm:IFT}. 

\begin{claim}
  \label{4.13}
  There exist $\delta>0$ and $c>0$ such that, for   all $u \in \ker \LL$ and $\om,\tau,\la \in \R$
  with $\|u\|_\infty+|\om-1|+|\tau-\tau_0|+|\la|<\delta$,
it holds
  $$
  \|(I-P)\left(I-\CC(\om,\la)-\DD(\om,\la)(\J(\om,\tau,\la)+\K(\la))\right)w\|_\infty \ge c\|w\|_\infty \mbox{ for all } w \in \W.
  $$
\end{claim}
\begin{subproof} We will follow ideas which are used to prove coercitivity estimates for singularly perturbed linear differential operators
	(see, e.g., \cite[Lemma 1.3]{Magnus} and \cite[Section 3]{RO}).
  
Suppose the contrary. Then there exist sequences $w_n \in \W$, $u_n \in \ker \LL$ and $(\om_n,\tau_n,\la_n) \in \R^3$
  such that
  \beq
  \label{one}
  \|w_n\|_\infty=1 \mbox{ for all } n \in \N,
  \ee
\beq
  \label{tozero1}
  \|u_n\|_\infty+|\om_n-1|+|\tau_n-\tau_0|+|\la_n| \to 0 \mbox{ as } n\to\infty
  \ee
  and
\beq
\label{tozero2}
\|(I-P)\left(I-\CC(\om_n,\la_n)-\DD(\om_n,\la_n)(\J(\om_n,\tau_n,\la_n)+\KK(\la_n))\right)w_n\|_\infty \to 0 \mbox{ as } n\to\infty.
\ee
We have to construct a contradiction.

For the sake of simpler writing we will use the following notation:
\begin{eqnarray*}
  &&\CC_n:=\CC(\om_n,\la_n),\; \DD_n:=\DD(\om_n,\la_n)(\J(\om_n,\tau_n,\la_n)+\KK(\la_n)),\\
  &&\EE_n:=(I-\CC_n)^{-1}(\DD_n+P(I-\CC_n-\DD_n)).
\end{eqnarray*}
Note that  the operators $\EE_n$ are well defined due to Claim \ref{C1}
from Section \ref{Fredholmness of the linearization}.
By assumption \reff{tozero2}, we have 
\beq
\label{tozero3}
\|(I-P)(I-\CC_n-\DD_n)w_n\|_\infty=\|(I-\CC_n)(I-\EE_n)w_n\|_\infty \to 0.
\ee
Moreover, because of \reff{I-C-1} it follows that $\|(I-\EE_n)w_n\|_\infty \to 0$ and, on the account of
\reff{bounded} and \reff{I-C-1}, that
\beq
\label{tozero4}
\|(I+\EE_n)(I-\EE_n)w_n\|_\infty=\|(I-\EE_n^2)w_n\|_\infty \to 0.
\ee

Let us show that the sequence $\EE_n^2w_n$ is bounded in the space $C^1_{2 \pi}$.
A straightforward calculation shows that
\beq
\label{Epred}
\EE_n^2=(I-\CC_n)^{-1}\left(\DD_n^2+\DD_n\CC_n(I-\CC_n)^{-1}\DD_n+\RR_n\right)
\ee
with
\beq
  \label{Rdef}
  \RR_n:=\DD_n(I-\CC_n)^{-1}P(I-\CC_n-\DD_n)+P(I-\CC_n-\DD_n)(I-\CC_n)^{-1}(\DD_n+P(I-\CC_n-\DD_n)).
\ee
From \reff{bounded}, \reff{Cdelta}, \reff{one} and \reff{Epred} follows that, in order to  show that $\EE_n^2w_n$ is bounded in $C^1_{2 \pi}$,
it suffices to show that the operators sequences $\DD_n^2$, $\DD_n\CC_n$ and $\RR_n$ are bounded with respect to the uniform operator norm
in $\LL(C_{2\pi};C^1_{2\pi})$. Let us start with
$$
\DD_n^2=\DD(\om_n,\la_n)(\J(\om_n,\tau_n,\la_n)+\KK(\la_n))\DD(\om_n,\la_n)(\J(\om_n,\tau_n,\la_n)+\KK(\la_n)).
$$
This sequence is bounded in $\LL(C_{2\pi};C^1_{2\pi})$ because of \reff{bounded}, \reff{sup1} and \reff{sup2}. Then consider
$$
\DD_n\CC_n=\DD(\om_n,\la_n)(\J(\om_n,\tau_n,\la_n)+\KK(\la_n))\CC(\om_n,\la_n).
$$
This sequence is bounded in $\LL(C_{2\pi};C^1_{2\pi})$ because of \reff{bounded}, \reff{sup1} and \reff{sup3}. And finally, the operator sequence $\RR_n$ 
is bounded in $\LL(C_{2\pi};C^1_{2\pi})$ because of \reff{bounded}, \reff{I-C-1}, \reff{Cdelta} and because the projection $P$ belongs to  $\LL(C_{2\pi};C^1_{2\pi})$
(cf. \reff{Pdef}).

Let us summarize: We showed that
the sequence $\EE_n^2w_n$ is bounded in $C^1_{2 \pi}$.
Because of the  Arzela-Ascoli Theorem, 
without loss of generality we may assume that this sequence 
converges in  $C_{2 \pi}$
to some function $w_* \in C_{2 \pi}$. 
Then \reff{tozero4} implies the convergence
\beq
\label{tozero5}
\|w_n-w_*\|_\infty \to 0.
\ee
In particular, $w_* \in \W$.

If we would have
\beq
\label{contra}
\LL w_*=(I-\CC-\DD(\J+\KK))w_*=0,
\ee
then it would follow $w_* \in \ker \LL \cap \W$ and, on the account of  \reff{split}, $w_*=0$, contradicting  \reff{one} and \reff{tozero5}. 
Hence, it remains to prove \reff{contra}.

In order to  prove \reff{contra}, we take arbitrary $w,h \in C_{2\pi}$ and calculate
\begin{eqnarray*}
  &&\langle (\CC_n-\CC)w,h\rangle\\
  &&=\frac{1}{2\pi}\int_0^{2\pi}\int_0^1\Big(\left[-c_1(x,0,\la_n)w_2(t-\om_nA(x,0,\la_n),0)+c_1(x,0,0)w_2(t-A(x,0,0),0)\right]h_1(t,x)\\
  &&\;\;\;\;\;\;\;\;+\left[c_2(x,1,\la_n)w_1(t+\om_nA(x,1,\la_n),1)-c_2(x,1,0)w_1(t+A(x,1,0),1)\right]h_2(t,x)\Big)dxdt\\
 &&=\frac{1}{2\pi}\int_0^{2\pi}\int_0^1\Big(\left[-c_1(x,0,\la_n)h_1(t+\om_nA(x,0,\la_n),x)+c_1(x,0,0)h_1(t+A(x,0,0),x)\right]w_2(t,0)\\
  &&\;\;\;\;\;\;\;\;+\left[c_2(x,1,\la_n)h_2(t-\om_nA(x,1,\la_n),x)-c_2(x,1,0)h_2(t-A(x,1,0),x)\right]w_1(t,1)\Big)dxdt.
\end{eqnarray*}
Hence,
$$
\lim_{n \to \infty}\sup_{\|w\|_\infty \le 1}\langle (\CC_n-\CC)w,h\rangle=0 \mbox{ for all } h \in C_{2\pi}.
$$
Similarly one shows  for all $h_1,h_2,h_3 \in C_{2\pi}$ the convergence
$$
\lim_{n \to \infty}\left(\sup_{\|w\|_\infty \le 1}\langle (\DD_n-\DD)w,h_1\rangle+
\sup_{\|w\|_\infty \le 1}\langle (\J_n-\J)w,h_2\rangle
+\sup_{\|w\|_\infty \le 1}\langle (\KK_n-\KK)w,h_3\rangle\right)=0.
$$
Therefore, we get from 
\reff{kerP}, \reff{one} and \reff{tozero3} that
$$
0=\lim_{n \to \infty}\langle(I-P)(I-\CC_n-\DD_n(\J_n+\KK_n))w_n,h\rangle=\langle(I-P)(I-\CC-\DD(\J+\KK))w_*,h\rangle
=\langle \LL w_*,h \rangle
$$
for all $h \in C_{2\pi}$, i.e., \reff{contra} is true.
\end{subproof}

We have  shown
 that Theorem \ref{thm:IFT} can be applied to equation  \reff{ex} in the setting \reff{setting}. This implies the following fact.
\begin{claim}
  \label{4.14}
  There exist $\eps>0$ and $\delta>0$
  such that for all $u \in \UU_\eps$ and $(\om,\tau,\la)\in \PP_\eps$ there is a unique solution $w=\hat{w}(u,\om,\tau,\la) \in \W$
  to \reff{ex} with $\|w\|_\infty <\delta$.
  Moreover, for all $k\in\N$ the partial derivatives $\partial_t^k\hat{w}(u,\om,\tau,\la)$ exist and belong to $C_{2\pi}$, and the map
  $(u,\om,\tau,\la) \in  \UU_\eps\times\PP_\eps\mapsto \partial_t^k\hat{w}(u,\om,\tau,\la) \in C_{2\pi}$ is $C^\infty$-smooth.
\end{claim}

In order to finish the  proof of Lemma \ref{4.10}
by using Claim \ref{4.14},
we have to show that for all $k \in \N$ it holds
\beq
\label{wsmooth}
\hat{w}(u,\om,\tau,\la) \in C^k_{2\pi} \mbox{ for all } (u,\om,\tau,\la) \in  \UU_\eps\times\PP_\eps
\ee
and that the map
\beq
\label{wsmooth1}
(u,\om,\tau,\la) \in  \UU_\eps\times\PP_\eps \mapsto \hat{w}(u,\om,\tau,\la) \in C^k_{2\pi} \mbox{ is $C^\infty$-smooth}.
\ee

Let us first prove \reff{wsmooth}. We use induction with respect to $k$. For  $k=0$ condition \reff{wsmooth} is true
because of Claim  \ref{4.14}.

In order to do the induction step we use that  $\hat{w}(u,\om,\tau,\la)\in \W^{k+1}$ (because of Claim  \ref{4.14}) and that
$\hat{w}(u,\om,\tau,\la)\in C^k_{2\pi}$ (because of the induction assumption), and we have to show that
$\hat{w}(u,\om,\tau,\la)\in C^{k+1}_{2\pi}$ (induction assertion).
It holds
$$
\hat{w}(u,\om,\tau,\la)=F(\hat{w}(u,\om,\tau,\la),u,\om,\tau,\la) \mbox{ for all } (u,\om,\tau,\la) \in  \UU_\eps\times\PP_\eps,
$$
where the map $F:\W \times   \UU_\eps\times\PP_\eps\to \W$ is defined by
\begin{eqnarray*}
  &&F(w,u,\om,\tau,\la)\\
  &&:=\CC(\om,\la)w+\DD(\om,\la)\B(u+w,\om,\tau,\la)+(I-P)(I-\CC(\om,\la))u-P\DD(\om,\la)\B(u+w,\om,\tau,\la).
\end{eqnarray*}
Hence, we have to show that
\beq
\label{Fprop}
F(w,u,\om,\tau,\la) \in C^{k+1}_{2\pi} \mbox{ for all } w \in \W^{k+1}\cap  C^k_{2\pi}
\mbox{ and } (u,\om,\tau,\la) \in  \UU_\eps\times\PP_\eps.
\ee
Obviously, for all $w \in C_{2\pi}$ and $ (u,\om,\tau,\la) \in  \UU_\eps\times\PP_\eps$ it holds $(I-P)(I-\CC(\om,\la))u \in  C^{l}_{2\pi}$
and $P\DD(\om,\la)\B(u+w,\om,\tau,\la) \in  C^{l}_{2\pi}$ for any $l \in \N$.
Hence, in order to prove \reff{Fprop},  it remains to show that
$$
\CC(\om,\la)w,\DD(\om,\la)\B(u+w,\om,\tau,\la)  \in C^{k+1}_{2\pi} \mbox{ for all } w \in \W^{k+1}\cap  C^k_{2\pi}
\mbox{ and } (u,\om,\tau,\la) \in  \UU_\eps\times\PP_\eps,
$$
or, the same, that
\beq
\label{Fprop1}
\left.
  \begin{array}{r}
    \partial_t[\CC(\om,\la)w],\partial_x[\CC(\om,\la)w],
\partial_t[\DD(\om,\la)\B(u+w,\om,\tau,\la)], \partial_x[\DD(\om,\la)\B(u+w,\om,\tau,\la)]
    \in C^{k}_{2\pi}\\
    \mbox{for all } w \in \W^{k+1}\cap  C^k_{2\pi}.
  \end{array}
  \right\}
  \ee
  Due to the definitions of $\B, \CC$ and $\DD$, given by the formulas
  \reff{BBdef}, \reff{Cdef} and \reff{Ddef}, it holds
  \begin{eqnarray*}
    &&\partial_t[\CC(\om,\la)w]=\CC(\om,\la)\partial_tw,\; \partial_t[\DD(\om,\la)w]=\DD(\om,\la)\partial_tw,\\
    &&\partial_x[\CC(\om,\la)w]=\widetilde{\CC}(\om,\la)w+\widehat{\CC}(\om,\la)\partial_tw,
\partial_x[\DD(\om,\la)w]=\widetilde{\DD}(\om,\la)w+\widehat{\DD}(\om,\la)\partial_tw\\
    &&\partial_t[\B(v,\om,\tau,\la)]=\partial_v\B(v,\om,\tau,\la)\partial_tv,
  \end{eqnarray*}
  where
  $$
  [\widetilde{\CC}(\om,\la)w](t,x):=
  \left[
  \begin{array}{c}
    -\partial_xc_1(x,0,\la)w_2(t+\om A(x,0,\la),0)\\
    \partial_xc_2(x,1,\la)w_1(t-\om A(x,1,\la),1)
  \end{array}
\right],
$$
 $$
  [\widehat{\CC}(\om,\la)w](t,x):=
  \left[
  \begin{array}{c}
    -\om\partial_xA(x,0,\la)c_1(x,0,\la)w_2(t+\om A(x,0,\la),0)\\
    -\om\partial_xA(x,1,\la)c_2(x,1,\la)w_1(t-\om A(x,1,\la),1)
  \end{array}
\right],
$$    
  $$
  [\widetilde{\DD}(\om,\la)w](t,x):=
  \left[
  \begin{array}{c}
    \displaystyle-\frac{w_1(t,x)}{a(x,\la)}-\int_0^x\frac{\partial_xc_1(x,\xi,\la)}{a(\xi,\la)}w_1(t+\om A(x,\xi,\la),0)d\xi\\
     \displaystyle-\frac{w_2(t,x)}{a(x,\la)}+\int_x^1\frac{\partial_xc_2(x,\xi,\la)}{a(\xi,\la)}w_2(t-\om A(x,\xi,\la),0)d\xi
  \end{array}
\right],
$$
 $$
  [\widehat{\DD}(\om,\la)w](t,x):=
\left[
  \begin{array}{c}
    \displaystyle-\frac{\om}{a(x,\la)}\int_0^x\frac{c_1(x,\xi,\la)}{a(\xi,\la)}
    w_1(t+\om A(x,\xi,\la),0)d\xi\\
      \displaystyle-\frac{\om}{a(x,\la)}\int_x^1\frac{c_2(x,\xi,\la)}{a(\xi,\la)}
      w_2(t-\om A(x,\xi,\la),0)d\xi
  \end{array}
\right].
$$
Hence, \reff{Fprop1} is true.

Now, let us prove \reff{wsmooth1}. Again, we use induction with respect to $k$. For  $k=0$, condition \reff{wsmooth1} is true 
due to Claim \ref{4.14}, again.
For the induction step, we proceed as above. We have to show that the map
\beq
\label{wsmooth2}
(u,\om,\tau,\la) \in  \UU_\eps\times\PP_\eps \mapsto \hat{w}(u,\om,\tau,\la) \in C^{k+1}_{2\pi} \mbox{ is $C^\infty$-smooth}
\ee
under the assumption that 
\beq
\label{wsmooth3}
(u,\om,\tau,\la) \in  \UU_\eps\times\PP_\eps \mapsto \hat{w}(u,\om,\tau,\la) \in \W^{k+1}\cap C^{k}_{2\pi} \mbox{ is $C^\infty$-smooth}.
\ee
Thanks to \reff{wsmooth3},
the maps 
\begin{eqnarray*}
  &&(u,\om,\tau,\la) \mapsto \CC(\om,\la)\partial_t\hat{w}(u,\om,\tau,\la) \in C^{k}_{2\pi},\\
   &&(u,\om,\tau,\la) \mapsto \widetilde{\CC}(\om,\la)\hat{w}(u,\om,\tau,\la) \in C^{k}_{2\pi},\\
   &&(u,\om,\tau,\la) \mapsto \widehat{\CC}(\om,\la)\hat{w}(u,\om,\tau,\la) \in C^{k}_{2\pi},\\
  &&(u,\om,\tau,\la) \mapsto \widetilde{\DD}(\om,\la)\B(u+\hat{w}(u,\om,\tau,\la),\om,\tau,\la)\in C^{k}_{2\pi},\\
  &&(u,\om,\tau,\la) \mapsto \widehat{\DD}(\om,\la)\partial_v\B(u+\hat{w}(u,\om,\tau,\la),\om,\tau,\la)
     (\partial_tu+\partial_t\hat{w}(u,\om,\tau,\la)) \in C^{k}_{2\pi}
\end{eqnarray*}
are $C^\infty$-smooth, which implies \reff{wsmooth2} as desired.
  
\begin{rem}
	\label{wprop}
	The uniqueness assertion of Lemma \ref{4.14} and equality \reff{Finv} yield
	\beq
	\label{winv}
	S_\vp\hat{w}(u,\om,\tau,\la)=\hat{w}(S_\vp u,\om,\tau,\la)
	\mbox{ for all } \vp \in \R, u \in \UU_\eps \mbox{ and } (\om,\tau,\la)\in \PP_\eps.
	\ee
	Moreover,    the uniqueness assertion of Lemma \ref{4.14} 
	along with equality $\B(0,\om,\tau,\la)=0$ implies that
	\beq
	\label{wnull}
	\hat{w}(0,\om,\tau,\la)=0
	\mbox{ for all } (\om,\tau,\la)\in \PP_\eps.
	\ee
	Finally, differentiating the identity
	$$
	(I-P)\left((I-\CC(\om,\la))(u+\hat{w}(u,\om,\tau,\la))-\DD(\om,\la)\B(u+\hat{w}(u,\om,\tau,\la)),\om,\tau,\la)\right)=0
	$$
	with respect to $u$ in $u=0$, $\om=1$, $\tau=\tau_0$ and $\la=0$,
	we conclude that $\LL\partial_u\hat{w}(0,1,\tau_0,0)=0$, i.e.,
	$\partial_u\hat{w}(0,1,\tau_0,0) \in \ker \LL \cap \W$, i.e.,
	\beq
	\label{pwnull}
	\partial_u\hat{w}(0,1,\tau_0,0)=0.
	\ee
\end{rem}

\section{The bifurcation equation}
\label{The bifurcation equation}
\setcounter{equation}{0}
In this section we substitute the solution $w=\hat{w}(u,\om,\tau,\la)$ to \reff{ex}  into \reff{in} and
solve the resulting
so-called bifurcation equation
\beq
\label{bifeq}
    P\left((I-\CC(\om,\la))(u+\hat{w}(u,\om,\tau,\la))-\DD(\om,\la)\B(u+\hat{w}(u,\om,\tau,\la)),\om,\tau,\la)\right)=0
    \ee
with respect to $\om \approx 1$ and $\tau \approx \tau_0$ for $u\approx 0$ and $\la \approx 0$.
The definition \reff{Pdef} of the projection $P$ shows that equation \reff{bifeq} is equivalent to
 \beq
\label{bifeq1}
    \langle(I-\CC(\om,\la))(u+\hat{w}(u,\om,\tau,\la))-\DD(\om,\la)\B(u+\hat{w}(u,\om,\tau,\la)),\om,\tau,\la),\A^*\vv_*\rangle=0.
    \ee
    By Lemma \ref{ker},  the variable $u \in \ker \LL$ in \reff{bifeq1} 
    can be replaced by 
    $\xi \in \C$, by using the ansatz
    \beq
    \label{ans}
    u=\mbox{Re}\,(\xi\vv_0)=\xi_1v_0^1-\xi_2v_0^2,\; \xi=\xi_1+i\xi_2,\; \xi_1,\xi_2 \in \R.
    \ee
    We, therefore, get the following equation:
    \begin{eqnarray}
    \label{bifeq2}
    F(\xi,\om,\tau,\la)&:=&
    \langle(I-\CC(\om,\la))(\mbox{Re}\,(\xi\vv_0)+\hat{w}(\mbox{Re}\,(\xi\vv_0),\om,\tau,\la)),\A^*\vv_*\rangle\nonumber\\
    &&-\langle\DD(\om,\la)\B(\mbox{Re}\,(\xi\vv_0)+\hat{w}(\mbox{Re}\,(\xi\vv_0),\om,\tau,\la)),\om,\tau,\la),\A^*\vv_*\rangle=0.
    \end{eqnarray}
    On the account of \reff{vnulldef} and \reff{v*def},  we have $S_\vp\vv_0=e^{i\vp}\vv_0$ and  $S_\vp\vv_*=e^{i\vp}\vv_*$ and, 
hence, $S_\vp \mbox{Re}\,(\xi\vv_0)=\mbox{Re}\,(e^{i\vp}\xi\vv_0)$. Now, \reff{inv1}, \reff{inv2} and \reff{winv}
      yield
      \beq
      \label{Finv1}
      e^{i\vp}F(\xi,\om,\tau,\la)=F(e^{i\vp}\xi,\om,\tau,\la).
      \ee
    Our task is, therefore, reduced  to determining all solutions $\xi \approx 0$, $\om \approx 1$, $\tau \approx \tau_0$ and $\la \approx 0$
      with real non-negative $\xi$. Since from now on $\xi$ is considered to be a real parameter, we redenote it by $\eps$. Equation \reff{bifeq2} then reads
        \begin{eqnarray}
    \label{bifeq3}
    G(\eps,\om,\tau,\la)&:=&
    \langle(I-\CC(\om,\la))(\eps v_0^1+\hat{w}(\eps v_0^1,\om,\tau,\la)),\A^*\vv_*\rangle\nonumber\\
    &&-\langle\DD(\om,\la)\B(\eps v_0^1+\hat{w}(\eps v_0^1,\om,\tau,\la)),\om,\tau,\la),\A^*\vv_*\rangle=0.
        \end{eqnarray}
        From $\B(0,\om,\tau,\la)=0$ and \reff{wnull} it follows that
         $G(0,\om,\tau,\la)\equiv 0$.
        This means that, to  solve \reff{bifeq3} with $\eps>0$, it suffices to solve
        the so-called scaled or restricted bifurcation equation
        \beq
         \label{bifeq4}
         H(\eps,\om,\tau,\la):=\frac{1}{\eps}G(\eps,\om,\tau,\la)=\int_0^1\partial_\eps G(s\eps,\om,\tau,\la)ds=0.
         \ee
      In particular, on the account of \reff{Aid1} and \reff{bifeq3} it holds
         \begin{eqnarray}
\label{k1}
&& H(\eps,1,\tau_0,0)\nonumber\\
&&=
\int_0^1\left\langle\left(\A-\partial_v\B(s\eps v_0^1+\hat{w}(s\eps v_0^1,1,\tau_0,0),1,\tau_0,0\right))  (I+\partial_u\hat{w}(s\eps v_0^1,1,\tau_0,0))v_0^1,\vv_*\right\rangle\, ds,
         \end{eqnarray}
and \reff{wnull} and \reff{pwnull} yield
\beq
\label{k1a}
H(0,\om,\tau_0,0)=
\langle\A\left(I-\CC(\om,0)
               -\DD(\om,0)\partial_v\B(0,\om,\tau_0,0)\right)v_0^1,\vv_*\rangle.
  \ee

 By \reff{wnull}, \reff{pwnull}, \reff{k1} and Lemma \ref{ker} we have $H(0,1,\tau_0,0)=\langle \LL v_0^1,\A^*\vv_*\rangle=0$.
         Hence, in order to  solve \reff{bifeq4} with respect to $\om \approx 1$ and $\tau \approx \tau_0$
         (for  $\eps \approx 0$  and $\la \approx 0$)
         by using the classical implicit function theorem we have to show that
         \beq
         \label{Jac}
         \left.\det \frac{\partial (\mbox{Re}\,H,\mbox{Im}\,H)}{\partial(\om,\tau)}\right|_{\eps=\la=0, \om=1, \tau=\tau_0} \ne 0.
         \ee

         Let us calculate the partial derivatives in the Jacobian in the left-hand side of \reff{Jac}.
        Due to  
\reff{LLdef}, \reff{k1a}
and Lemma  \ref{ker} 
we have
  \begin{eqnarray}
           \partial_\om H(0,1,\tau_0,0)
           &=&\left.\frac{d}{d\om}\langle\A(\om,0)\left(I-\CC(\om,0)
               -\DD(\om,0)\partial_v\B(0,\om,\tau_0,0)\right)v_0^1,\vv_*\rangle\right|_{\om=1}\nonumber\\
           &=&\left.\frac{d}{d\om}\langle\left(\A(\om,0)
              -\J(\om,\tau_0,0)-\KK(0))\right)v_0^1,\vv_*\rangle\right|_{\om=1}\nonumber\\
             & =&\langle\left(\partial_\om\A(1,0)
              -\partial_\om\J(1,\tau_0,0)\right)v_0^1,\vv_*\rangle.
           \label{dom}
  \end{eqnarray}
  Similarly one gets
  \beq
  \label{dtau}
  \partial_\tau H(0,1,\tau_0,0)=-\langle\partial_\tau\J(1,\tau_0,0)v_0^1,\vv_*\rangle.
  \ee
    On the other hand, \reff{J12def} implies that for $k=1,2$
  \begin{eqnarray*}
&&[\partial_\om\J_k(1,\tau_0,0)v](t,x)=-\frac{\tau_0}{2}b_4^0(x)\int_0^x\frac{\d_tv_1(t-\tau_0,\xi)-\d_tv_2(t-\tau_0,\xi)}{a_0(\xi)}d\xi,\\
    &&[\partial_\tau\J_k(1,\tau_0,0)v](t,x)=-\frac{1}{2}b_4^0(x)\int_0^x\frac{\d_tv_1(t-\tau_0,\xi)-\d_tv_2(t-\tau_0,\xi)}{a_0(\xi)}d\xi
  \end{eqnarray*}
Moreover,  \reff{AAdef} yields   $\partial_\om\A(1,0)\bf{\vv_0}=\partial_t \bf\vv_0$.
  Hence, from \reff{vnulldef} it follows that
  $$
  [\partial_\om\A(1,0)v_0^1](t,x)=\mbox{Re}\,\d_t\vv_0(t,x)=
      \mbox{Re}\,\left(e^{it}\left[
      \begin{array}{c}
        -u_0(x)+ia_0(x)u_0'(x)\\
        -u_0(x)-ia_0(x)u_0'(x)
        \end{array}
      \right]\right)
    $$
    and, for $k=1,2$, that
 \begin{eqnarray*}
   [\partial_\om\J_k(1,\tau_0,0)v_0^1](t,x)&=&\mbox{Re}\,[\partial_\om\J_k(1,\tau_0,0)\vv_0](t,x)=
   -\frac{\tau_0}{2}b_4^0(x)\,\mbox{Re}\,\left(ie^{it}\int_0^x\frac{v_{01}(\xi)-v_{02}(\xi)}{a_0(\xi)}d\xi\right)\\
   &=&\tau_0b_4^0(x)\,\mbox{Im}\left(e^{i(t-\tau_0)}u_0(x)\right),
      \end{eqnarray*}
      where
      \beq
      \label{comp}
      v_{01}(x):=iu_0+a_0u_0',\; v_{02}(x):=iu_0-a_0u_0'
      \ee
      are the components of the vector function $v_0$ (cf. \reff{vnulldef}).
      Analogously one gets
      $
    [\partial_\tau\J_k(1,\tau_0,0)v_0^1](t,x)=b_4^0(x)\,\mbox{Im}\left(e^{i(t-\tau_0)}u_0(x)\right).
 $
 We insert this into \reff{dom} and \reff{dtau} and get
 \begin{eqnarray*}
   &&\partial_\om H(0,1,\tau_0,0)\\
   &&=\frac{1}{2\pi}\int_0^{2\pi}\int_0^1\left(\mbox{Re}\,\left(e^{it}\left[
      \begin{array}{c}
        -u_0+ia_0u_0'\\
        -u_0-ia_0u_0'
        \end{array}
   \right]\right)-\tau_0b_4^0\mbox{Im}\left(e^{i(t-\tau_0)}
   \left[\begin{array}{c}
           u_0\\u_0
           \end{array}
           \right]\right)\right)\cdot \left(e^{it}\left[
      \begin{array}{c}
        u_*+iU_*\\
        u_*-iU_*
      \end{array}
   \right]\right)\,dxdt\\
   &&=\frac{1}{2}\int_0^1
      \left[
      \begin{array}{c}
        (-1+i\tau_0b_4^0e^{-i\tau_0}))u_0+ia_0u_0'\\
        (-1-i\tau_0b_4^0e^{-i\tau_0}))u_0-ia_0u_0'
      \end{array}
   \right]\cdot\left[
      \begin{array}{c}
        u_*+iU_*\\
        u_*-iU_*
      \end{array}
   \right]\,dx\\
   &&=\int_0^1\left((-1+i\tau_0b_4^0e^{-i\tau_0})u_0\overline{u_*}+a_0u_0'\overline{U_*}\right)\,dx
   =\int_0^1\left((-1+i\tau_0b_4^0e^{-i\tau_0})\overline{u_*}-(a_0\overline{U_*})'\right)u_0\,dx.
 \end{eqnarray*}
 Here we used \reff{adbc}. By \reff{s3} and \reff{transvid}, it holds
 \beq
 \label{dom1}
\partial_\om H(0,1,\tau_0,0)=\int_0^1\left(-2+i\tau_0b_4^0(x)e^{-i\tau_0}-ib_5^0(x)\right)u_0(x)\overline{u_*(x)}dx=i.
  \ee
  Similarly,
  \begin{eqnarray}
   \mbox{Re}\,\partial_\tau H(0,1,\tau_0,0)
  & =&-\frac{1}{2\pi}\mbox{Re}\int_0^{2\pi}\int_0^1b_4^0\,\mbox{Im}\left(e^{i(t-\tau_0)}
   \left[\begin{array}{c}
           u_0\\u_0
           \end{array}
           \right]\right)\cdot \left(e^{it}\left[
      \begin{array}{c}
        u_*+iU_*\\
        u_*-iU_*
      \end{array}
   \right]\right)dxdt\nonumber\\
 &=&\mbox{Re}\left(ie^{-i\tau_0}\int_0^1b_4^0(x)u_0(x)\overline{u_*(x)}dx\right)=\rho.
 \label{dtau1}
  \end{eqnarray}
  Hence, assumption {\bf(A2)} implies that
  $$
  \det\left[
    \begin{array}{cc}
      \mbox{Re}\,\partial_\om H(0,1,\tau_0,0) & \mbox{Im}\,\partial_\om H(0,1,\tau_0,0\\
      \mbox{Re}\,\partial_\tau H(0,1,\tau_0,0) & \mbox{Im}\,\partial_\tau H(0,1,\tau_0,0
    \end{array}
  \right]=
  - \mbox{Re}\,\partial_\tau H(0,1,\tau_0,0)=-\rho \ne 0,
  $$
i.e., \reff{Jac} is true.

  Now, the classical implicit function theorem can be applied to solve \reff{bifeq4} with respect to $\om \approx 1$ and $\tau \approx \tau_0$  for $\eps \approx 0$  
  and  $\la \approx 0$. We, therefore, conclude that
there exist $\eps_0>0$ and $C^\infty$-smooth functions $\hat{\om},\hat{\tau}:[-\eps_0,\eps_0]^2\to \R$ with
  $\hat{\om}(0,0)=1$ and $\hat{\tau}(0,0)=\tau_0$ such that $(\eps,\om,\tau,\la) \approx (0,1,\tau_0,0)$ is a solution to \reff{bifeq4}
  if and only if
  \beq
  \label{bifeqsol}
  \om=\hat{\om}(\eps,\la),\quad \tau=\hat{\tau}(\eps,\la).
  \ee
  Moreover, equality \reff{Finv1} implies that $F(-\xi,\om,\tau,\la)=-F(\xi,\om,\tau,\la)$, i.e., $H(-\eps,\om,\tau,\la)=H(\eps,\om,\tau,\la)$.
  Thus, $\hat{\om}(-\eps,\la)=\hat{\om}(\eps,\la)$ and  $\hat{\tau}(-\eps,\la)=\hat{\tau}(\eps,\la)$.
  This yields \reff{bifdir}.
  Now, by \reff{wnull} and \reff{pwnull},
  the corresponding solutions to \reff{FOS}, where $\om$ and $\tau$ are given by \reff{bifeqsol},
  read
  \begin{eqnarray*}
  v=\eps [\hat{v}(\eps,\la)](t,x)&:=&\eps v_0^1(t,x)+[\hat{w}(\eps v_0^1,\hat{\om}(\eps,\la),\hat{\tau}(\eps,\la),\la)](t,x)\\
  &=&\eps\mbox{Re}\left(e^{it}\left[
      \begin{array}{c}
        iu_0(x)+a_0(x)u_0'(x)\\
        iu_0(x)-a_0(x)u_0'(x)
      \end{array}
    \right]\right)+o(\eps),
  \end{eqnarray*}
  where $o(\eps)/\eps \to 0$ for $\eps \to 0$ uniformly with respect to $\la \in [-\eps_0,\eps_0]$. Finally, we take into account
  \reff{udef} and conclude  that the solutions to \reff{problem}, corresponding to  $\om$ and $\tau$,
  are defined by
  $$
  u=\eps [\hat{u}(\eps,\la)](t,x):=\frac{\eps}{2}\int_0^x\frac{[\hat{v}_1(\eps,\la)](t,\xi)-[\hat{v}_2(\eps,\la)](t,\xi)}{a(\xi,\la)}d\xi
  =\eps \mbox{Re}\left(e^{it}u_0(x)\right)+o(\eps),
  $$
  which proves \reff{uas}.

\section{The bifurcation direction}
\label{sec:dir}
\setcounter{equation}{0}
\subsection{Proof of Theorem \ref{thm:dir}}

Differentiating the identity
$
\mbox{Re}\,H(\eps,\hat{\om}(\eps,0),\hat{\tau}(\eps,0),0)\equiv 0
$
two times with respect to $\eps$ at $\eps=0$ and taking into account \reff{bifdir}, \reff{dom1} and \reff{dtau1},   we get
\beq
\label{ka}
\mbox{Re}\,\partial_\eps^2 H(0,1,\tau_0,0)=\rho \partial_\eps^2\hat{\tau}(0,0).
\ee
Furthermore,   
\reff{k1}, \reff{wnull} and \reff{pwnull}  yield the equality
\begin{eqnarray}
\label{pH}
&&\partial^2_\eps H(0,1,\tau_0,0)\nonumber\\
&&=-\langle \partial^3_v\B(0,1,\tau_0,0)(v_0^1,v_0^1,v_0^1)+2\partial^2_v\B(0,1,\tau_0,0)(v_0^1,
\partial^2_u\hat{w}(0,1,\tau_0,0)(v_0^1,v_0^1)),\vv_*\rangle.
\end{eqnarray}

Now we  use the special structure \reff{rest}, \reff{rest1} of the nonlinearity $b$ as it is assumed in Theorem~\ref{thm:dir}.
It follows from \reff{rest}, \reff{rest1}, \reff{Bdef} and \reff{BBdef} that $\partial^2_v\B(0,1,\tau_0,0)=0$. Moreover, for $j=1,2$, it holds
\begin{eqnarray*}
  &&[\partial^3_v\B_j(0,1,\tau_0,0)(v_0^1,v_0^1,v_0^1)](t,x)=[\partial^3_vB(0,1,\tau_0,0)(v_0^1,v_0^1,v_0^1)](t,x)\\
  &&=\beta_1^0(x)[J_0v_0^1](t,x)^3+\beta_2^0(x)[J_0v_0^1](t-\tau_0,x)^3+\beta_3^0(x)[Kv_0^1](t,x)^3+\beta_4^0(x)[K_0v_0^1](t,x)^3.
     \end{eqnarray*}
 Furthermore, \reff{Jdef}, \reff{Kdef}, \reff{vnulldef} and \reff{comp} yield    
\begin{eqnarray*}
  [J_0v_0^1](t,x)&=&\mbox{Re}\,[J_0\vv_0](t,x)=\frac{1}{2}\mbox{Re}\left(e^{it}\int_0^x\frac{v_{01}(\xi)-v_{02}(\xi)}{a_0(\xi)}d\xi\right)=
   \mbox{Re}\,(e^{it}u_0(x)),\\  
  \mbox{[$Kv_0^1$]}(t,x)&=&\mbox{Re}\,[K\vv_0](t,x)=\frac{1}{2}\mbox{Re}\,(e^{it}(v_{01}(x)+v_{02}(x)))=-\mbox{Im}\,(e^{it}u_0(x)),\\
  \mbox{[$K_0v_0^1$]}(t,x)&=&\d_x[J_0v_0^1](t,x)=\mbox{Re}\,(e^{it}u'_0(x)). 
     \end{eqnarray*}
     Therefore,
\begin{eqnarray*}     
  &&[\partial^3_vB(0,1,\tau_0,0)(v_0^1,v_0^1,v_0^1)](t,x)\\
  &&=\frac{1}{8}\left(\beta_1^0(x)(e^{it}u_0(x)+e^{-it}\overline{u_0(x)})^3
     +\beta_2^0(x)(e^{i(t-\tau_0)}u_0(x)+e^{-i(t-\tau_0)}\overline{u_0(x)})^3\right.\\
  &&\left.
     \;\;\;\;\;+i\beta_3^0(x)(e^{it}u_0(x)-e^{-it}\overline{u_0(x)})^3+\beta_4^0(x)(e^{it}u'_0(x)+e^{-it}\overline{u'_0(x)})^3\right)\\
  &&=\frac{1}{8}e^{3it}\left((\beta_1^0(x)+\beta_2^0(x)e^{-3\tau_0}+i\beta_3^0(x))u_0(x)^3+\beta_4^0(x)u_0'(x)^3\right)\\
  &&\;\;\;\;\;+\frac{3}{8}e^{it}\left((\beta_1^0(x)+\beta_2^0(x)e^{-i\tau_0}+i\beta_3^0(x))u_0(x)^2\overline{u_0(x)}+\beta_4^0(x)u_0'(x)^2
\overline{u'_0(x)}
     \right)\\
   &&\;\;\;\;\;+\frac{3}{8}e^{-it}\left((\beta_1^0(x)+\beta_2^0(x)e^{i\tau_0}+i\beta_3^0(x))u_0(x)\overline{u_0(x)}^2+\beta_4^0(x)u_0'(x)
\overline{u'_0(x)}^2
      \right)\\
   &&\;\;\;\;\;+\frac{1}{8}e^{-3it}\left((\beta_1^0(x)+\beta_2^0(x)e^{3i\tau_0}+i\beta_3^0(x))\overline{u_0(x)}^3+\beta_4^0(x)
\overline{u'_0(x)}^3
     \right).
\end{eqnarray*}
Inserting this into \reff{ka} and \reff{pH}, we end up with the equality
\begin{eqnarray*}   
&&\partial_\eps^2\hat{\tau}(0,0)=\frac{1}{\rho}\mbox{Re}\,\partial_\eps^2 H(0,1,\tau_0,0)
  =-\frac{1}{\rho}\langle \partial^3_v\B(0,1,\tau_0,0)(v_0^1,v_0^1,v_0^1),v^1_*\rangle\\
  &&=\frac{1}{2\pi\rho}\int_0^{2\pi}
     \int_0^1
     \left[
     \begin{array}{c}
       [\partial^3_vB(0,1,\tau_0,0)(v_0^1,v_0^1,v_0^1)](t,x)\\
       \displaystyle [\partial^3_vB(0,1,\tau_0,0)](v_0^1,v_0^1,v_0^1)(t,x)
     \end{array}
  \right]
  \cdot\mbox{Re}\left(e^{it}
  \left[
  \begin{array}{c}
       u_*(x)+iU_*(x)\\
         u_*(x)-iU_*(x)
     \end{array}
  \right]
  \right)dxdt\\
  &&=\frac{1}{2\pi\rho}\mbox{Re}\int_0^{2\pi}
     \int_0^1[\partial^3_vB(0,1,\tau_0,0)(v_0^1,v_0^1,v_0^1)](t,x)e^{-it}\overline{u_*(x)}dxdt\\
     &&=\frac{3}{8\rho} \mbox{Re}\left(\int_0^1\left(
    (\beta^0_1(x)+\beta^0_2(x)e^{-i\tau_0} +i\beta^0_3(x))|u_0(x)|^2u_0(x)+\beta_4(x)
    |u'_0(x)|^2u'_0(x)\right)\overline{u_*(x)}dx\right).
\end{eqnarray*}
This is exactly the desired formula in Theorem \ref{thm:dir} with $\sigma=1$ (cf. \reff{transvid}).

\subsection{Example}
	Let us consider problem  \reff{eq:1.1}, \reff{eq:1.2} with
	$$
\tau_0= \frac{\pi}{2}, \; 	a_0(x)=\frac{4}{\pi^2}, \; b_3^0(x)=b_6^0(x)=0,\;  b_4^0(x)=b_5^0(x)=c(x) \mbox{ for all }  x \in [0,1]
	$$
	and  a smooth function $c:[0,1] \to \R$.
	The function $u(x)= \sin\frac{\pi x}{2}$ then solves \reff{evp}  with $\mu=i$ and  \reff{ad}, and
	the choice $u_0(x)=u_*(x)=\sin\frac{\pi x}{2}$ gives
	$$
	\sigma=-\int_0^1c(x) \left(\sin\frac{\pi x}{2}\right)^2 dx+i\int_0^1\left(2-i\frac{\pi}{2}e^{-i\frac{\pi}{2}}c(x)\right)\left(\sin\frac{\pi x}{2}\right)^2 dx,\;
	\rho=-\frac{1}{|\sigma|^2}\int_0^1c(x)\left(\sin\frac{\pi x}{2}\right)^2 dx.
	$$
	Hence, if $\int_0^1c(x)\sin\left(\frac{\pi x}{2}\right)^2 dx\not=0$, then all assumptions of Theorem \ref{thm:hopf} are satisfied.
	If, additionally,  $\beta_4^0(x)=0$ for all $x \in [0,1]$, then
	\begin{eqnarray*}
		\frac{8}{3}|\sigma|^2\rho \,\partial_\eps\hat{\tau}(0,0)&=&
		-\int_0^1c(x)\left(\sin\frac{\pi x}{2}\right)^2 dx\int_0^1\beta_1^0(x)\left(\sin\frac{\pi x}{2}\right)^4 dx\\
		&&+\int_0^1\left(2-\frac{\pi}{2}c(x)\right)\left(\sin\frac{\pi x}{2}\right)^2 dx
		\int_0^1(\beta_3^0(x)-\beta_2^0(x))\left(\sin\frac{\pi x}{2}\right)^4 dx.
	\end{eqnarray*}
	Therefore, if this number is positive, then the Hopf bifurcation is supercritical.

\section{Other boundary conditions}
\label{other}
\renewcommand{\theequation}{{\thesection}.\arabic{equation}}
\setcounter{equation}{0}
The results of Theorems \ref{thm:hopf} and \ref{thm:dir} can be 
extended to  other 
than \reff{eq:1.2} 
boundary conditions,
for example,  for two Dirichlet, or two Robin (in particular,  Neumann), or for periodic boundary conditions.
However, in those cases the transformation \reff{vdef} 
is not appropriate anymore. 
Instead of \reff{vdef}, the following 
transformation can be used:
\beq
\label{vdef1}
\left.
\begin{array}{rcl}
  v_1(t,x)&=&\om (\d_tu(t,x)-u(t,x))+a(x,\la) \d_xu(t,x),\\
  \;v_2(t,x)&=&\om (\d_tu(t,x)-u(t,x))-a(x,\la) \d_xu(t,x).
\end{array}
\right\}
\ee
The inverse transformation  is then given by
\beq
\label{udef1}
u(t,x)=\frac{e^t}{2\om}\left(\int_0^te^{-s}(v_1(s,x)+v_2(s,x))ds-\frac{1}{1-e^{-2\pi}}\int_0^{2\pi}e^{-s}(v_1(s,x)+v_2(s,x))ds\right).
\ee
More precisely, if $u \in C^2_{2\pi}(\R\times[0,1])$ satisfies the second order differential equation in \reff{problem}
and if  $v \in C^1_{2\pi}(\R\times[0,1];\R^2)$ is defined by \reff{vdef1}, then
$v$ satisfies the first order system
\beq
\label{FOS2}
\left.
  \begin{array}{l}
    \om\d_tv_1(t,x)-a(x,\la)\d_xv_1(t,x)+\om v_2(t,x)=[B(v,\om,\tau,\la)](t,x)\\
    \om\d_tv_2(t,x)+a(x,\la)\d_xv_2(t,x)+\om v_1(t,x)=[B(v,\om,\tau,\la)](t,x),
  \end{array}
\right\}
\ee
with
$$
\begin{array}{rcl}
  [B(v,\om,\tau,\la)](t,x)&:=&\displaystyle b\left(x,\la,[J v](t,x)/\om,[J v](t-\om\tau,x)/\om, [(J+K)v](t,x),[K_\la v](t,x)\right)\nonumber\\ [2mm]
                          && \displaystyle  -\frac{1}{2}\d_xa(x,\la)(v_1(t,x)-v_2(t,x)),\\ [3mm]
  \mbox{$[J v]$}(t,x)&:=&\displaystyle \frac{e^t}{2}\left(\int_0^te^{-s}(v_1(s,x)+v_2(s,x))ds-\frac{1}{1-e^{-2\pi}}
                          \int_0^{2\pi}e^{-s}(v_1(s,x)+v_2(s,x))ds\right)
\end{array}
$$
and with operators $K$ and $K_\la$ defined in \reff{Kdef}.
And vice versa, if  $v \in C^1_{2\pi}(\R\times[0,1];\R^2)$ satisfies \reff{FOS2} and if  $u \in C^1_{2\pi}(\R\times[0,1])$,
is defined by \reff{udef1}, then  $u$ is $C^2$-smooth and  satisfies the differential equation in \reff{problem}.
Note that till now  no boundary conditions were used, but
  only the periodicity in time.

Now, for definiteness,   suppose that
 $u$ satisfies the  Dirichlet boundary conditions
\beq
\label{ubc}
u(t,0)=u(t,1)=0.
\ee
Then, accordingly to \reff{vdef1}, the function $v$ satisfies 
\beq
\label{vbc}
v_1(t,0)+v_2(t,0)=v_1(t,1)+v_2(t,1)=0.
\ee
In this case
the system \reff{IE} of partial integral equations reads
\beq
\label{IE1}
\left.
  \begin{array}{l}
    v_1(t,x)+c_1(x,0,\la)v_2(t+\om A(x,0,\la),0)\\
    =\displaystyle-\int_0^x\frac{c_1(x,\xi,\la)}
                {a(\xi,\la)}[\B_1(v,\om,\tau,\la)](t+\om A(x,\xi,\la),\xi)d\xi,\\
    v_2(t,x)+c_2(x,1,\la)v_1(t-\om A(x,1,\la),1)\\
    =\displaystyle\int_x^1\frac{c_2(x,\xi,\la)}
                {a(\xi,\la)}[\B_2(v,\om,\tau,\la)](t-\om A(x,\xi,\la),\xi)d\xi,
  \end{array}
\right\}
\ee
where the  operator $\B$ is now defined by
\beq
\label{Bdef1}
\left.
\begin{array}{rcl}
    [\B_1(v,\om,\tau,\la)](t,x)&:=&[B(v,\om,\tau,\la)](t,x)-b_1(x,\la)v_1(t,x)-\om v_2(t,x),\\
 \displaystyle[\B_2(v,\om,\tau,\la)](t,x)&:=&[B(v,\om,\tau,\la)](t,x)-b_2(x,\la)v_2(t,x)-\om v_1(t,x),
\end{array}
\right\}
\ee
while the functions $b_1,b_2,c_1,c_2$ and $A$ are introduced in Sections \ref{sec:FOS} and  \ref{sec:Inteq}.

More exactly, if $v \in  C_\pi(\R\times[0,1];\R^2)$  satisfies \reff{FOS2} and if $\d_tv$ exists and is continuous, then the function $u$, defined by \reff{udef1},
is $C^2$-smooth and satisfies the differential equation in \reff{problem} and the boundary condition \reff{ubc}.

The system \reff{IE1} can be written, again, as an operator equation of the type \reff{abstract} with operators
$\CC_1$ and $\DD$ as in Section \ref{Lyapunov-Schmidt Procedure}, with $\CC_2$ slightly changed (cf. \reff{Cdef}) to 
$$
[\CC_2(\om,\la)v](x,t)=-c_2(x,1,\la)v_1(t-\om A(x,1,\la),1)
$$
and with operator $\B$ from \reff{Bdef1}.

Now we can proceed as in Sections \ref{Lyapunov-Schmidt Procedure}-\ref{sec:dir}.
Specifically, the linearization of operator $\B$ in $v=0$ is, again,
a sum of a partial integral operator and a pointwise operator with zero diagonal part
(cf. \reff{diffB}):
$$
\partial_v\B(0,\om,\tau,\la)=\J(\om,\tau,\la) + \KK(\om,\la)
$$
with, for $k=1,2$,
$$
[\J_k(\om,\tau,\la)v](t,x)=
\left(\frac{b_3(x,\la)}{\om}+b_5(x,\la)\right)[J v](t,x)+\frac{b_4(x,\la)}{\om} [J v](t-\om\tau,x)
  $$
(with functions $b_3,b_4,b_5$ from \reff{bdef1})  and
  $$
 \displaystyle [\KK_1(\om,\la)v](t,x)=(b_2(x,\la)-\om)v_2(t,x),\;  \displaystyle [\KK_2(\om,\la)v](t,x)=(b_1(x,\la)-\om)v_1(t,x).
 $$
The definition \reff{vnulldef} of the function $v_0$ has to be changed to
$$
v_0(x):=
\left[
  \begin{array}{c}
    (i-1)u_0(x)+a_0(x)u'_0(x)\\
    (i-1)u_0(x)-a_0(x)u'_0(x)
  \end{array}
\right],
$$
and similarly for  $\vv_0$, $v_0^1$ and $v_0^2$. The definitions \reff{v*def} of $v_*$, $\vv_*$, $v_*^1$ and $v_*^2$
stay the same. The functions  $v_*^1$ and $v_*^2$ satisfy the boundary conditions \reff{vbc} because 
$$
v_{*1}(x)+v_{*2}(x)=2u_*(x)=0 \mbox{ in } x=0,1.
$$
Here $u_0$ and $u_*$ are eigenfunctions to the eigenvalue problems \reff{evp} (with $\mu=i$ and $\tau=\tau_0$)
and \reff{ad}, where in both eigenvalue problems the boundary conditions are changed to \reff{ubc}.
With these eigenfunctions, the formulas for $\sigma$ and $\rho$ in {\bf(A3)} and the formula
for $\partial_\eps^2\hat{\tau}(0,0)$
in Theorem \ref{thm:dir} remain unchanged.

\section*{Acknowledgments}
Irina Kmit was supported by the VolkswagenStiftung Project ``From Modeling and Analysis to Approximation''.
Lutz Recke was supported by the DAAD program ``Ostpartnerschaften''.

\end{document}